\documentclass[a4paper]{amsart}
\usepackage[english]{babel}
\usepackage[utf8]{inputenc}
\usepackage[T1]{fontenc}
\usepackage{lmodern}

\usepackage[shortlabels]{enumitem}
\setlist[enumerate]{label={\arabic*.}}
\setlist[description]{font=\normalfont\itshape}

\usepackage[dvipsnames]{xcolor}
\usepackage[backref]{hyperref}
\hypersetup{
  colorlinks=true,
  linkcolor=Maroon,
  citecolor=OliveGreen
}

\usepackage{todonotes}

\usepackage{cleveref}
\usepackage[nopatch=footnote]{microtype}

\usepackage{setspace}
\usepackage{amsmath,amsfonts,amsthm,mathtools,amssymb}
\usepackage{aliascnt}

\def\Z{\mathbf{Z}}

\def\C{\mathbf{C}}
\def\F{\mathbf{F}}
\def\E{\mathbf{E}}

\DeclareMathOperator{\tr}{tr}

\DeclareMathOperator{\End}{End}
\DeclareMathOperator{\im}{im}

\DeclareMathOperator{\Ad}{Ad}
\DeclareMathOperator{\diag}{diag}

\DeclareMathOperator{\Prob}{Prob}
\DeclareMathOperator{\diam}{diam}
\DeclareMathOperator{\Id}{Id}
\DeclareMathOperator{\Sym}{Sym}

\newcommand{\GL}{\mathrm{GL}}
\newcommand{\SL}{\mathrm{SL}}
\newcommand{\SU}{\mathrm{SU}}

\newcommand{\slalg}{\mathfrak{sl}}

\newcommand{\PSL}{\mathrm{PSL}}
\newcommand{\AGL}{\mathrm{AGL}}

\newcommand{\rr}{\mathsf{R}}
\newcommand{\cc}{\mathsf{C}}

\newcommand{\diamp}[1][]{\diam^{#1}_+}

\newcommand{\lalg}{\mathfrak{g}}
\newcommand{\aalg}{\mathfrak{a}}
\newcommand{\mon}{\mathrm{mon}}

\newtheorem{theorem}{Theorem}[section]
\newtheorem*{theorem*}{Theorem}
\newaliascnt{lemma}{theorem}
\newtheorem{lemma}[lemma]{Lemma}
\aliascntresetthe{lemma}
\newaliascnt{proposition}{theorem}
\newtheorem{proposition}[proposition]{Proposition}
\aliascntresetthe{proposition}
\newaliascnt{corollary}{theorem}
\newtheorem{corollary}[corollary]{Corollary}
\aliascntresetthe{corollary}
\newaliascnt{conjecture}{theorem}

\aliascntresetthe{conjecture}
\newtheorem*{conjecture*}{Conjecture}

\theoremstyle{definition}
\newaliascnt{question}{theorem}
\newtheorem{question}[question]{Question}
\aliascntresetthe{question}
\newtheorem*{question*}{Question}
\newaliascnt{definition}{theorem}
\newtheorem{definition}[definition]{Definition}
\aliascntresetthe{definition}
\newaliascnt{example}{theorem}
\newtheorem{example}[example]{Example}
\aliascntresetthe{example}
\newtheorem*{example*}{Example}

\theoremstyle{remark}
\newaliascnt{remark}{theorem}
\newtheorem{remark}[remark]{Remark}
\aliascntresetthe{remark}

\theoremstyle{plain}
\newtheorem{theoremx}{Theorem}

\crefname{theoremx}{Theorem}{Theorems}
\Crefname{theoremx}{Theorem}{Theorems}

\crefname{theorem}{theorem}{theorems}
\Crefname{theorem}{Theorem}{Theorems}
\crefname{lemma}{lemma}{lemmas}
\Crefname{lemma}{Lemma}{Lemmas}
\crefname{proposition}{proposition}{propositions}
\Crefname{proposition}{Proposition}{Propositions}
\crefname{corollary}{corollary}{corollaries}
\Crefname{corollary}{Corollary}{Corollaries}
\crefname{question}{question}{questions}
\Crefname{question}{Question}{Questions}
\crefname{remark}{remark}{remarks}
\Crefname{remark}{Remark}{Remarks}
\crefname{definition}{definition}{definitions}
\Crefname{definition}{Definition}{Definitions}
\crefname{example}{example}{examples}
\Crefname{example}{Example}{Examples}
\crefname{conjecture}{conjecture}{conjectures}
\Crefname{conjecture}{Conjecture}{Conjectures}

\title[Additive diameters and covering complexity]{Additive diameters and covering complexity of irreducible representations}

\author{Urban Jezernik}
\address{Urban Jezernik, Faculty of Mathematics and Physics, University of Ljubljana, Jadranska 21, 1000 Ljubljana, Slovenia; Institute of Mathematics, Physics, and Mechanics, Jadranska 19, 1000 Ljubljana, Slovenia}
\email{urban.jezernik@fmf.uni-lj.si}

\author{Špela Špenko}
\address{Špela Špenko, Département de Mathématique, Université Libre de Bruxelles, Campus de la
Plaine CP 213, Bld du Triomphe, B-1050 Bruxelles, Belgium}
\email{spela.spenko@ulb.be}

\thanks{UJ was supported by the Slovenian Research Agency program P1-0222 and grants J1-50001, J1-70033. ŠŠ was supported by a MIS grant from the National Fund for Scientific Research
(FNRS) and an ARC grant from the Université Libre de Bruxelles.}

\begin{document}
\baselineskip=13pt 

\begin{abstract}
Let a group $G$ act linearly on a finite-dimensional complex vector space $V$.
The group-additive diameter of a subspace $U \leq V$ is the least number of translates of $U$ whose sum is all of $V$.
Counting dimensions, it is at least $\dim V / \dim U$.
We show that when $G$ is compact and $V$ is irreducible, the diameter of every nonzero subspace is at most $\lceil (\dim V / \dim U) \ln \dim V \rceil$, so the trivial bound is correct up to a logarithmic factor.
We measure the discrepancy by the covering complexity $\cc(V)$, the largest ratio between the diameter of any subspace and its trivial lower bound, so that $1 \leq \cc(V) \leq 2 + \ln \dim V$, and we determine where in this range various representations lie.
Every irreducible representation of $\SL_2(\C)$ has $\cc(V) = 1$.
The logarithm can be genuinely present along families of symmetric and exterior powers of $\SL_n(\C)$ with $n$ varying, and it is present for finite Heisenberg groups and $2$-transitive groups of small order such as $\PSL_2(\F_p)$.
The complexity is bounded above by a constant on the conjugation representations of $\SL_n(\C)$ and on the representations $\Sym^k \C^3$ of $\SL_3(\C)$.
On the other hand, every fixed connected reductive group has a family of irreducible representations whose complexities tend to at least the dimension of the flag variety.
Finally, for the Lie algebra $\slalg_3(\C)$ acting on $\Sym^k \C^3$, the monomial diameter with respect to $\Sym^k X$ for a plane $X \leq \C^3$ is optimal, while the corresponding $\SL_3(\C)$ diameter is not.
\end{abstract}

\maketitle


\section{Introduction}\label{sec:intro}

\subsection{Additive diameters}
Let $G$ be a group with a linear representation $\rho$ on a finite-dimensional complex vector space $V$.
How efficiently do the translates of a subspace fill out the representation? 
We measure the efficiency by the \emph{group-additive diameter} of $V$ with respect to a nonzero subspace $U \leq V$,
\[
  \diamp[G](V, U) = \min\bigl\{m \mid \rho(g_1) \cdot U + \dots + \rho(g_m) \cdot U = V \text{ for some } g_1, \dots, g_m \in G\bigr\}.
\]
This is the least number of translates of $U$ whose sum is all of $V$, infinite when no number of them suffices.
In other words, every element of $V$ is a sum of $\diamp[G](V, U)$ vectors, each taken from some translate of $U$.
When $V$ is irreducible the diameter is finite for every nonzero $U$, since the sum of all translates of $U$ is a nonzero submodule. Counting dimensions gives the trivial lower bound
\[
  \diamp[G](V, U) \geq \Bigl\lceil \frac{\dim V}{\dim U} \Bigr\rceil.
\]
The central question is how large the diameter can be compared to this trivial bound, and how this depends on the group, the representation, and the subspace.

We ran into these diameters in \cite{JS25} while studying noncommutative analogues of the classical Waring problem \cite{VW02}. 
Instead of writing every integer as a sum of few $k$-th powers, the noncommutative version replaces integers by matrices and $k$-th power by a fixed noncommutative polynomial, and asks to write every matrix as a sum of few values of the polynomial (see \cite{BM26} for a recent survey).
The set of values of a noncommutative polynomial has a special property, namely it is invariant under conjugation by the group of invertible matrices.
Here is how we can make use of this property.
If $f \colon W \to V$ is a polynomial map that is equivariant for a group $G$ acting linearly on both sides, then we can differentiate $f$ to get linear maps $D_w f \colon W \to V$ for $w \in W$ that inherit transformation properties under $G$. This led us in \cite[Theorem 7.8]{JS25} to the inequality
\[
\diam_+(V, \im f) \leq 2 \cdot \diamp[G](V, \im(D_w f)),
\]
where the left-hand side is the least number of summands with $\im f + \dots + \im f = V$.
Group-additive diameters therefore bound various Waring-type diameters.
In particular, the noncommutative version above corresponds to invertible matrices acting by conjugation on matrices, with $f$ a noncommutative polynomial.

Apart from some concrete representations tied to the applications above that we studied in \cite{JS25}, the behaviour of group-additive diameters remains largely unexplored.
This paper is a step towards a general theory.
We prove that for every irreducible unitary representation of every compact group, the diameter of every subspace is within a logarithmic factor of the trivial bound.
We show that this logarithm cannot be removed in general, and we begin to map out where it is present and where it is not.

\subsection{Growth and completion}\label{subsec:expansion}
Our approach to studying group-additive diameters is guided by an analogy with growth in finite groups.
There one might start with a generating set $S$ of a finite simple group $G$ of Lie type and ask how many products of elements of $S$ are needed to write every element of $G$.
The answer comes in two stages.
If $S$ is small, then products of copies of $S$ grow exponentially in size \cite{Hel08,BGT11,PS16} until they become very large, and at that point Gowers' trick \cite{Gow08,NP11} says that three copies of such a large set already cover $G$.
\emph{Growth} handles the small sets, and \emph{completion} handles the large ones.

We would wish for such a two-stage process in our problem as well.
In the first stage we add translates of $U$ one at a time, and we want the dimension of the sum to increase by a definite factor at each step.
In the second stage the sum already occupies a constant fraction of the dimension, and we want a bounded number of further translates to reach all of $V$.

A variant of the first stage has been studied under the name of \emph{dimension expanders} (a linear analogue of expander graphs).
We say that endomorphisms $A_1, \ldots, A_k$ of a finite-dimensional complex vector space $V$ form an $\epsilon$-dimension expander for some $\epsilon > 0$ if
\[
  \dim (U + A_1 U + \dots + A_k U) \geq (1 + \epsilon) \dim U
\]
for every subspace $U \leq V$ of dimension at most $\dim V/2$.
Thus a fixed finite set of maps enlarges every small subspace by a fixed factor in dimension.
Lubotzky and Zelmanov \cite{LZ08} proved using Kazhdan's property (T) that there are constants $k$ and $\epsilon > 0$, independent of $\dim V$, such that every $V$ has a set of $k$ endomorphisms that form an $\epsilon$-dimension expander.
Reineke \cite{Rei24} has recently obtained sharp existence results using quiver representations, determining for each fixed $k$ the optimal $\epsilon$.
See \cite{LQWWZ25} for a general overview of dimension expanders.

Our setting differs from that of dimension expanders in the following way.
A dimension expander is a fixed finite set of maps that has to work for every subspace at once.
We, on the other hand, are free to choose the translates after seeing the subspace.
In return we can prove  more. The following is what produces the growth stage in our setting.

\begin{theoremx}[\Cref{thm:first-moment}]\label{thm:A}
Let $K$ be a compact group with an irreducible unitary representation $\rho$ on a finite-dimensional complex vector space $V$, and let $U, W \leq V$.
Then there exists $g \in K$ with
\[
  \dim\bigl(W + \rho(g) \cdot U\bigr) \geq \dim W + \Bigl(1 - \frac{\dim W}{\dim V}\Bigr)\dim U .
\]
\end{theoremx}

A single translate therefore expands any subspace, while the expansion factor decays with the density of $W$ in $V$.
The statement passes to complex reductive groups, where the good translates form a nonempty Zariski open set.
The diameter over a reductive group is the same as over its maximal compact subgroup, so we move freely between the two settings.

Iterating the above growth gives our main general upper bound on the diameter.

\begin{theoremx}[\Cref{thm:covering}]\label{thm:B}
Let $K$ be a compact group with an irreducible unitary representation on $V$, put $N = \dim V \geq 2$ and $u = \dim U$ for a nonzero $U \leq V$. Then
\[
  \diamp[K](V, U) \leq \min\Bigl\{ \Bigl\lceil \frac{N \ln N}{u} \Bigr\rceil, \quad \frac{N(1 + \ln u)}{u} + 1, \quad N - u + 1 \Bigr\}.
\]
\end{theoremx}

So the diameter is always within a factor $\ln \dim V$ of the trivial bound.
Let us give a name to this discrepancy.
For a nonzero $U \leq V$ put
\[
  \rr(V, U) = \frac{\diamp[K](V, U)}{\lceil \dim V / \dim U \rceil}
  \qquad \text{and} \qquad
  \cc(V) = \max_{0 \neq U \leq V} \rr(V, U).
\]
Thus $\cc(V)$ is the largest factor by which the diameter of a subspace of $V$ exceeds its trivial bound, and $\cc(V) = 1$ means that every subspace of $V$ has an optimal diameter.
\Cref{thm:B} therefore gives
\[
  1 \leq \cc(V) \leq 2 + \ln \dim V .
\]
The logarithm here appears because of running the growth stage all the way to the end.
In order to get a better bound, we need to stop the growth stage when the sum of translates already occupies a constant fraction of $V$, and then finish with a bounded number of further translates.
We prove that if every subspace of $V$ of density at least some fixed $\theta < 1$ has diameter at most $C_0$, then $\cc(V)$ is bounded by a constant depending only on $\theta$ and $C_0$.
So the question becomes whether such a completion step exists, and the answer depends on the representation.
The rest of the paper studies this question in several situations, and the rest of the introduction describes what we found.

\subsection{Optimal diameters}
The best one can hope for is $\cc(V) = 1$, so that every subspace of $V$ has an optimal diameter.
There is a strong property that forces this.
Counting dimensions, two subspaces can meet in dimension as low as $\max(0, \dim W + \dim U - \dim V)$, and this is typically much smaller than what \Cref{thm:first-moment} gives.
We call $V$ \emph{transversal} if for every pair of subspaces some translate of $U$ meets $W$ in this smallest possible dimension.
A transversal representation has $\cc(V) = 1$, and we show that every irreducible representation of $\SL_2(\C)$ is transversal.
This gives a second proof of the optimality of all diameters over $\SL_2(\C)$ from \cite{JS25}, and a more conceptual one.
The standard representations of the symmetric groups are transversal as well.
Transversality is strictly stronger than $\cc(V) = 1$, and it fails badly in the examples of the next subsection.

\subsection{Logarithmic complexity}
The logarithm in \Cref{thm:covering} cannot be removed in general.

\begin{theoremx}[\Cref{thm:heis,thm:2trans}]\label{thm:C}
Let the inner supremum below range over all irreducible unitary representations $V$ of all compact groups with $\dim V = N$. Then
\[
  \limsup_{N \to \infty} \ \sup_{\dim V = N} \frac{\cc(V)}{\ln N} = 1 .
\]
\end{theoremx}

As a concrete example, now for connected reductive groups, the representations 
\begin{equation} \label{eq:sym_ext}
\cc_{\SL_{k+1}(\C)}\big(\Sym^k \C^{k+1}\big),
\
\cc_{\SL_{2m}(\C)}\big(\Lambda^m \C^{2m}\big)
\ \geq \ 
\ln \dim V/(4\ln 2).
\end{equation}
We produce these and other examples of logarithmic complexity from the following two mechanisms.

The first one uses subspaces with a lot of structure.
The model case is $V = \Sym^k \C^n$ under $\SL_n(\C)$ with $U = \Sym^k X$ for a hyperplane $X$.
The translates of $U$ are the symmetric powers of other hyperplanes.
The diameter turns out to be
\[
  \diamp[\SL_n(\C)](\Sym^k \C^n, \Sym^k X) = k + 1
  \qquad \text{for every } k \geq 1,
\]
whereas the trivial bound can be as small as $2$.
Exterior powers of a hyperplane behave in the same way.
The latter representations can be restricted to the symmetric groups, showing that $S_n$ has representations with both $\cc = 1$ on its standard module and $\cc$ logarithmic on a suitable exterior power of that module.

The second mechanism needs the group to be small compared to $\dim V$.
An example is the Heisenberg group of order $p^3$ in its Schrödinger representation on $\C[\Z/p]$.
The translates of the coordinate subspace indexed by a subset $S \subseteq \Z/p$ are the coordinate subspaces indexed by the translates $S + t$. The diameter is thus exactly the least number of translates of $S$ that cover $\Z/p$, a quantity studied in additive combinatorics.
General results of Bollobás, Janson and Riordan \cite{BJR11} give $\cc(\C[\Z/p]) = (1 + o(1))\ln p$.
One might think that the Heisenberg group is to blame, since it offers only $p$ translates of each coordinate subspace.
It is not.
We show that the same happens for the deleted permutation module of any $2$-transitive group on $n$ points of order at most $\exp(n^{o(1)})$, which includes the affine $2$-transitive groups and the projective actions of $\PSL_d(\F_q)$.
For $\PSL_2(\F_p)$ acting on the projective line this module is the Steinberg representation.

\subsection{Fixed groups}
In the examples above the group changes with the representation, so it remains possible that the complexity is bounded for every fixed connected group.
We do not know whether this is the case.
What we do show is that the complexity of the conjugation representation of $\SL_n(\C)$ is bounded by a constant independent of $n$, and that the same holds for the representations $\Sym^k \C^3$ of $\SL_3(\C)$.
In the latter family, the dimension of the representation grows while the group stays fixed.
The proof follows the two-stage scheme of growth and completion.
We have not been able to treat the other irreducible representations of $\SL_3(\C)$ in the same way.

On the other hand, we can also produce lower bounds for complexities of a fixed connected group.

\begin{theoremx}[\Cref{cor:flagbound}]\label{thm:D}
Let $G$ be a connected reductive group with $\dim G/B > 0$ and let $\lambda$ be a regular dominant weight. Then
\[
  \liminf_{k \to \infty} \cc\bigl(V(k\lambda)\bigr) \geq \dim G/B.
\]
\end{theoremx}

\subsection{Lie algebras}
A representation of a complex algebraic group $G$ can be differentiated to a representation of its Lie algebra $\lalg$.
In \cite[Section 6]{JS25} we asked what the additive diameter looks like in this setting.
We translate subspaces using the monomials $\rho(x_1) \cdots \rho(x_r)$ with $x_i \in \lalg$.
The \emph{monomial Lie-additive diameter} $\diamp[\lalg,\mon](V, U)$ is the least number of monomials $m_1, \ldots, m_d$ with $m_1 \cdot U + \dots + m_d \cdot U = V$.
Translating by a monomial does not increase dimension, so the trivial lower bound for the diameter persists here as well.

It is unclear how the group-additive and the Lie-additive diameters are related.
In \cite[Questions 6.6 and 6.7]{JS25} we wondered whether irreducible representations of Lie algebras always have optimal monomial diameters, and whether the monomial diameter is always at most the group diameter.
We can support both with the following new example. 
Let $\slalg_3(\C)$ act on $\Sym^k \C^3$. Let $X$ be a plane in $\C^3$. Then
\[
  \diamp[\slalg_3(\C),\mon](\Sym^k \C^3, \Sym^k X) = \Bigl\lceil \frac{k+2}{2} \Bigr\rceil
  \qquad \text{for every } k \geq 1 .
\]
This is the same pair as in \eqref{eq:sym_ext}, but now the group is replaced by its Lie algebra.
The group needs $k+1$ translates, asymptotically twice the trivial bound, while the Lie algebra attains the trivial bound.

\subsection{Reader's guide}

We begin with the growth stage.
In \Cref{sec:intersect} we prove \Cref{thm:A} by an averaging argument, and we record that the diameter over a reductive group agrees with the diameter over a maximal compact subgroup.
In \Cref{sec:covering} we iterate the intersection bound to obtain \Cref{thm:B}, introduce the covering complexity, and record the explicit application to equivariant morphisms.
The next three sections locate several examples of representations at the two ends of the range $1 \leq \cc(V) \leq 2 + \ln \dim V$.
In \Cref{sec:optimal} we use transversality to show that $\cc(V) = 1$ for the standard representations of the symmetric groups and for the irreducible representations of $\SL_2(\C)$.
In \Cref{sec:degree} we obtain logarithmic complexity from structured subspaces using the Borel--Weil theorem, which gives \Cref{thm:D} and the families in \eqref{eq:sym_ext}, and in \Cref{sec:small} we obtain it from small groups, namely the Heisenberg groups and $2$-transitive groups of small order, which completes the proof of \Cref{thm:C}.
In \Cref{sec:bounded} we turn to the completion stage and use it to bound the complexity of the conjugation representations of $\SL_n(\C)$ and of the representations $\Sym^k \C^3$ of $\SL_3(\C)$.
Finally, in \Cref{sec:lie} we compute the monomial Lie-additive diameter of $\Sym^k X$ in $\Sym^k \C^3$ over $\slalg_3(\C)$.

\subsection{AI disclosure}

This paper began with a mistake of ours.
An earlier version of \cite{JS25} claimed that a certain additive diameter is linear in $1/\epsilon$ (Proposition 5.3 there).
Ernesto Ingrosso found an error in the proof, and he then found a counting argument over the symmetric group recovering the linear bound \cite{Ing26}.
Put to Claude Opus 5, that argument turned out to be an averaging argument in disguise, and the model produced \Cref{thm:A}, generalizing it to all compact groups.
Over many further prompts, Claude Opus 5 and Fable 5 produced the rest of the results in this paper.
The models also drafted much of the text.
A preliminary version of the paper was then put to ChatGPT 5.6 Sol, which found several inaccuracies and simplified some of the arguments.

We had been thinking about the underlying theory and the questions addressed in this paper for quite some time, well before the advent of modern AI tools.
Our own contribution to the present work has been to direct the models, mostly by asking for examples that would test the sharpness of the bounds, to verify every proof in detail by hand, to expand on the denser parts of the proofs, to simplify some of the more complicated arguments, and to organize and rewrite the paper and develop its overall narrative.
None of the proofs here originated with us, and the examples were worked out by the models.
That said, we take full responsibility for what the paper contains.

The train we boarded here is fast and exciting. We do not know where it is going, but we cannot shake the uneasy feeling that we are about to plunge into an abyss.

\subsection{Acknowledgements}

We thank Ernesto Ingrosso for sharing \cite{Ing26} that this paper grew out of.

\section{Intersections of translates}\label{sec:intersect}

Throughout this section $K$ is a compact group with a finite-dimensional irreducible unitary representation $\rho$ on $V$, and $N = \dim V$.
We work over $\C$ with a Hermitian inner product, and we write $P_U$ for the orthogonal projection onto a subspace $U$.
Let $\pi$ be the associated representation of $K$ on $\End(V)$ given by
\[
  \pi(g)A = \rho(g) A \rho(g)^{-1},
\]
which is unitary for the Hilbert--Schmidt form $\langle A, B \rangle = \tr(AB^*)$.
As a representation, $\End(V) \cong V \otimes V^*$.
For subspaces $U, W \leq V$ the function we study is
\[
  X(g) = \tr\bigl(P_W \rho(g) P_U \rho(g)^{-1}\bigr) = \bigl\langle \pi(g) P_U,\ P_W \bigr\rangle.
\]

\subsection{A trace bound for intersections}
We need the following elementary bound on the trace of a product of two projections.
This trace is related to the principal angles between the subspaces, going back to Jordan \cite{Jor1875} and Björck and Golub \cite[Theorem 1]{BG73}, and in that framework the inequality below is folklore.
It also appears in the proof of Li, Qiao, Wigderson, Wigderson and Zhang \cite[Theorem 1.9(4)]{LQWWZ25}, and a weaker version is Lubotzky and Zelmanov \cite[Lemma 2.2]{LZ08}.
We give the short proof because we need exactly this form.

\begin{lemma}\label{lem:trace}
Let $P, Q$ be orthogonal projections on a finite-dimensional Hermitian space. Then
\[
  \dim\bigl(\im P \cap \im Q\bigr) \leq \tr(PQ).
\]
\end{lemma}

\begin{proof}
Since $P^2 = P$ and the trace is cyclic, $\tr(PQ) = \tr(P^2 Q) = \tr(PQP)$, and the operator $PQP$ is positive semidefinite because $\langle PQPv, v \rangle = \langle QPv, Pv \rangle = \|QPv\|^2 \geq 0$.

Let $Z = \im P \cap \im Q$ and $m = \dim Z$.
Choose an orthonormal basis $z_1, \dots, z_m$ of $Z$ and extend it to an orthonormal basis of the whole space.
For $z \in Z$ we have $Pz = Qz = z$, and hence
\[
  \langle PQPz, z \rangle = \langle Qz, z \rangle = 1.
\]
So the first $m$ diagonal entries of $PQP$ in this basis equal $1$, while the remaining ones are nonnegative by positivity.
Therefore $\tr(PQ) = \tr(PQP) \geq m$.
\end{proof}

\subsection{The first moment}
Here irreducibility enters, and it enters through Schur's lemma.
The average of a projection over all its translates commutes with everything, hence is a scalar, hence is determined by its trace.
Nothing else about the group is used.

\begin{theorem}\label{thm:first-moment}
Let $U, W \leq V$ be subspaces. Then there exists $g \in K$ with
\[
  \dim\bigl(W \cap \rho(g) \cdot U\bigr) \leq \frac{\dim U \cdot \dim W}{N}.
\]
\end{theorem}

\begin{proof}
Let $dg$ denote the normalized Haar measure on $K$ and consider the averaging operator
\[
  \E(A) = \int_K \pi(g) A \,dg \qquad \text{for } A \in \End(V).
\]
For every $h \in K$ we have $\pi(h)\E(A) = \E(A)$ by invariance of Haar measure, so $\E(A)$ lies in the commutant of $\rho(K)$.
Since $\rho$ is irreducible, Schur's lemma gives $\E(A) \in \C \cdot \Id_V$, and comparing traces, which $\E$ preserves, we obtain
\[
  \E(A) = \frac{\tr A}{N} \Id_V.
\]
Apply this to $A = P_U$, whose trace is $\dim U$.
As $\rho$ is unitary, $\pi(g)P_U$ is precisely the orthogonal projection onto $\rho(g) \cdot U$, and therefore
\[
  \int_K X(g) \,dg = \tr\bigl(P_W \E(P_U)\bigr) = \frac{\dim U \cdot \dim W}{N}.
\]
Some $g \in K$ therefore attains at most this average value, and for such a $g$ \Cref{lem:trace} gives the claim.
\end{proof}

\begin{remark}\label{rem:provenance}
The averaging identity $\E(P_U) = (\dim U/N)\Id$ is classical.
It is also the device behind \cite[Proposition 2.1]{LZ08}.
\Cref{lem:trace} is classical as well, as discussed above.
What we have not found in the literature is the explicit conclusion of \Cref{thm:first-moment} itself.
\end{remark}

The statement transfers to complex reductive groups by the unitary trick, and there the conclusion holds not merely for some $g$ but for a generic one.

\begin{corollary}\label{cor:reductive}
Let $G$ be a complex reductive linear algebraic group with an irreducible rational representation $\rho$ on $V$, and let $U, W \leq V$. Then
\[
  \dim\bigl(W \cap \rho(g) \cdot U\bigr) \leq \frac{\dim U \cdot \dim W}{\dim V}
\]
for all $g$ in a nonempty Zariski open subset of $G$.
\end{corollary}

\begin{proof}
Let $K$ be a maximal compact subgroup of $G$, which is Zariski dense in $G$, and note that $V$ carries a $K$-invariant Hermitian inner product.
Any $K$-submodule of $V$ is then a $G$-submodule by Zariski density, so $V$ is irreducible as a representation of $K$ and \Cref{thm:first-moment} applies.
Put
\[
  r = \Bigl\lfloor \frac{\dim U \cdot \dim W}{\dim V} \Bigr\rfloor,
\]
so that \Cref{thm:first-moment} produces $g \in K$ with $\dim(W \cap \rho(g) \cdot U) \leq r$, a dimension being an integer.
Finally, the condition $\dim(W \cap \rho(g) \cdot U) \leq r$ is Zariski open in $g$, since it says that a certain matrix built from bases of $W$ and $\rho(g) \cdot U$ has rank at least $\dim U + \dim W - r$.
It is satisfied by some element of $K$, hence by a nonempty open subset of $G$.
\end{proof}

The same open-rank argument identifies the two additive diameters that a reductive group carries, the one taken over the group itself and the one taken over a maximal compact subgroup.

\begin{lemma}\label{lem:compact-algebraic}
Let $G$ be a complex reductive group, let $K \leq G$ be a maximal compact subgroup, and let $V$ be a rational $G$-module.
Then $\diamp[G](V, U) = \diamp[K](V, U)$ for every $U \leq V$.
In particular the common value does not depend on the choice of $K$.
\end{lemma}

\begin{proof}
Since $K \subseteq G$, any tuple of $K$-translates spanning $V$ is a tuple of $G$-translates, so $\diamp[G](V, U) \leq \diamp[K](V, U)$.
Conversely, suppose $g_1 U + \dots + g_m U = V$ for some $(g_1, \dots, g_m) \in G^m$.
The set of tuples $(h_1, \dots, h_m) \in G^m$ with $h_1U + \dots + h_mU = V$ is Zariski open, again by a maximal-rank condition, and we have just seen that it is nonempty.
As $K^m$ is Zariski dense in $G^m$, it meets this open set, and therefore $\diamp[K](V, U) \leq m$.
\end{proof}

\begin{example}
\Cref{thm:first-moment} fails without irreducibility for trivial reasons.
Let $K$ act on $V = U \oplus U'$ with $U, U'$ nonzero submodules, and take $W = U$.
Then $W \cap \rho(g) \cdot U = U$ for all $g \in K$, and hence
\[
\dim(W \cap \rho(g) \cdot U) = \dim U
>
\frac{(\dim U)^2}{\dim V}.
\]
\end{example}

\begin{example} \label{ex:heisenberg_12}
The bound in \Cref{thm:first-moment} can be attained for all $g$.
Here is an example.
Let $\zeta = e^{2\pi i/12}$, and let $S, C$ act on $V = \C[\Z/12]$, with basis $(e_j)_{j \in \Z/12}$, by
\[
  S e_j = e_{j+1}
  \qquad \text{and} \qquad
  C e_j = \zeta^{j} e_j .
\]
The group $H_{12} = \langle S, C \rangle$ is a Heisenberg group over $\Z/12$, see \Cref{subsec:fin-heis} for details.
Every element of $H_{12}$ has the form $\zeta^{c}S^{a}C^{b}$ with $a,b,c \in \Z/12$, and these are pairwise distinct.
The space $V$ is an irreducible representation of $H_{12}$ of dimension $12$: the eigenvalues $\zeta^j$ of $C$ are pairwise distinct, so every invariant subspace is spanned by a subset of the basis, and stability under $S$ makes that subset a union of $\Z/12$-orbits under translation, hence empty or everything.

For a nonempty $A \subseteq \Z/12$ let $U_A = \langle e_j \mid j \in A \rangle$ be the coordinate subspace on $A$.
Take
  \[
    A = \{0,4,8\}
    \qquad \text{and} \qquad
    B = \{0,3,6,9\}
  \]
and put $U = U_A$ and $W = U_B$.
Scalars act trivially on subspaces and $C^{b}$ is diagonal, so $C^{b}U_A = U_A$, while $S^{a}U_A = U_{A+a}$.
Hence
  \[
    W \cap g \cdot U = U_{B \cap (A+a)}
    \qquad \text{for } g = \zeta^{c}S^{a}C^{b}.
  \]
Since $3$ and $4$ are coprime, the intersection $B \cap (A+a)$ is a singleton for every $a \in \Z/12$.
Therefore
  \[
    \dim \bigl( W \cap g \cdot U \bigr)
    = 1
    = \frac{\dim U \cdot \dim W}{\dim V}
    \qquad \text{for every } g \in H_{12}.
  \]
\end{example}

\section{Covering and the complexity of a representation}\label{sec:covering}

\subsection{The covering theorem}
The estimates above are statements about one step.
Iterating the first moment in the orthogonal complement turns it into a covering theorem, and the mechanism is simply that a constant factor of the codimension is destroyed at every step.
We refer to this iteration as the \emph{growth process}.

\begin{lemma}\label{lem:growth}
Let $K$ be a compact group with an irreducible unitary representation on $V$, put $N = \dim V$, and let $U \leq V$ be nonzero with $\alpha = \dim U/N$.
Then there are $g_1, g_2, \ldots \in K$ such that the subspaces $W_k = \sum_{i \leq k} \rho(g_i) \cdot U$ satisfy
\begin{equation}\label{eq:decay}
  \dim W_k^\perp
  < N e^{-\alpha k} \quad \text{for all } k \geq 1 .
\end{equation}
\end{lemma}

\begin{proof}
Set $W_0 = 0$.
Given $W_k$, apply \Cref{thm:first-moment} to the pair of subspaces $W_k^\perp$ and $U^\perp$, which produces $g_{k+1} \in K$ with
\[
  \dim\bigl(W_k^\perp \cap \rho(g_{k+1}) \cdot U^\perp\bigr) \leq \dim W_k^\perp \cdot \frac{\dim U^\perp}{N} = (1 - \alpha) \dim W_k^\perp ,
\]
and let $W_{k+1} = W_k + \rho(g_{k+1}) \cdot U$.
Since $\rho$ is unitary we have $\rho(g) \cdot U^\perp = (\rho(g) \cdot U)^\perp$, and therefore
\[
  W_{k+1}^\perp = W_k^\perp \cap \bigl(\rho(g_{k+1}) \cdot U\bigr)^\perp = W_k^\perp \cap \rho(g_{k+1}) \cdot U^\perp .
\]
Thus $\dim W_{k+1}^\perp \leq (1 - \alpha) \dim W_k^\perp$, and it follows by induction from $\dim W_0^\perp = N$ that $\dim W_k^\perp \leq N(1 - \alpha)^k$ for all $k \geq 1$.
The claim now follows from the fact that $1 - \alpha < e^{-\alpha}$ for $0 < \alpha < 1$, while for $\alpha = 1$ the left-hand side is zero.
\end{proof}

Running the growth process all the way to the bottom is wasteful, however.
Once the codimension has dropped below $\dim U$, a single translate could in principle finish the job, and in any case it can always be made to gain at least one dimension.
Balancing the two is what produces the middle term below, and it is the term that matters when $U$ is small.
The hypothesis $N \geq 2$ only excludes the one-dimensional representations, whose diameter is $1$.

\begin{theorem}\label{thm:covering}
Let $K$ be a compact group with an irreducible unitary representation on $V$, put $N = \dim V \geq 2$ and $u = \dim U$ for a nonzero subspace $U \leq V$. Then
\[
  \diamp[K](V, U) \leq \min\Bigl\{ \Bigl\lceil \frac{N \ln N}{u} \Bigr\rceil, \quad \frac{N(1 + \ln u)}{u} + 1, \quad N - u + 1 \Bigr\}.
\]
The same bounds hold for an irreducible rational representation of a complex reductive algebraic group, with generic translates.
\end{theorem}

\begin{proof}
For the third bound, let $W_r = \rho(g_1) \cdot U + \dots + \rho(g_r) \cdot U$ be a sum of $r$ translates of $U$ and suppose $W_r \neq V$.
Then $\rho(g) \cdot U \not\subseteq W_r$ for some $g \in K$, since otherwise the sum of all translates of $U$ would lie in $W_r$, whereas that sum is a nonzero submodule and hence all of $V$ by irreducibility.
So the next translate can always be chosen to gain a dimension,
\begin{equation}\label{eq:gain}
  \dim\bigl(W_r + \rho(g) \cdot U\bigr) \geq \dim W_r + 1 .
\end{equation}
Starting from $W_1 = U$ of dimension $u$, at most $N - u$ further translates are needed, which is the third bound.

We now turn to the other two bounds.
Put $\alpha = u/N$ and let $(W_k)$ be the growth process of \Cref{lem:growth}, so that $W_k$ is a sum of $k$ translates of $U$.

For the first bound, take $k = \lceil N \ln N / u \rceil$, so that $\alpha k \geq \ln N$.
Then $\dim W_k^\perp < N e^{-\alpha k} \leq N e^{-\ln N} = 1$ by \eqref{eq:decay}.
A dimension is an integer, so $W_k = V$.

For the second bound, we combine the two mechanisms.
When $u = 1$ it follows from the third bound, so assume $\alpha N = u \geq 2$.
Run the growth process for $k_1 = \lceil \ln(\alpha N)/\alpha \rceil$ steps, and observe from \eqref{eq:decay} that the codimension then satisfies
\[
  \dim W_{k_1}^\perp < N e^{-\alpha k_1} \leq N e^{-\ln(\alpha N)} = \frac{1}{\alpha}.
\]
We finish with at most $\lceil 1/\alpha \rceil - 1$ further steps, each gaining a dimension by \eqref{eq:gain}.
Since $\lceil x \rceil \leq x + 1$ and $\lceil x \rceil - 1 \leq x$, the total is at most
\[
  \frac{\ln(\alpha N)}{\alpha} + 1 + \frac{1}{\alpha} = \frac{N(1 + \ln u)}{u} + 1,
\]
using $\alpha N = u$, as desired.

Finally, let $G$ be complex reductive with an irreducible rational representation on $V$, and let $K$ be a maximal compact subgroup.
As in the proof of \Cref{cor:reductive}, $V$ is irreducible as a representation of $K$, so for each of the three bounds the compact case produces a tuple $(g_1, \dots, g_k) \in K^k$ with $\sum_i \rho(g_i) \cdot U = V$.
This condition asks a matrix built from bases of the spaces $\rho(g_i) \cdot U$ to have maximal rank, so it is Zariski open in $G^k$ and holds for a nonempty open set of tuples.
\end{proof}

\begin{remark}\label{rem:haar}
Over a connected compact group the good translates are of full measure, not merely generic.
Let $G$ be connected reductive with maximal compact subgroup $K$, and let
$Z \subseteq G$ be the complement of the open set produced by \Cref{cor:reductive}.
Since $K$ is connected and Zariski dense in $G$, $Z \cap K$ is a proper closed
real analytic subset of the connected manifold $K$, hence of measure zero.
So Haar almost every $g \in K$ satisfies the bound of \Cref{thm:first-moment},
and the same argument applied to $K^k$ shows Haar almost every tuple works as well.
\end{remark}

\begin{corollary}\label{cor:small}
For every $u_0$ there is $C = C(u_0)$, and one may take $C = 2 + \ln u_0$, such that every irreducible unitary representation $V$ of a compact group $K$ satisfies
\[
  \diamp[K](V, U) \leq C \frac{\dim V}{\dim U}
\]
for all nonzero $U \leq V$ with $\dim U \leq u_0$.
Moreover, we may take $C(1) = 1$, meaning that the diameter is optimal for one-dimensional subspaces.
\end{corollary}

\begin{proof}
Put $N = \dim V$ and $u = \dim U$, and note there is nothing to prove for $N = 1$.
The middle bound of \Cref{thm:covering} together with $1 \leq N/u$ gives at most $(2 + \ln u)N/u$ translates, and $\ln u \leq \ln u_0$.
When $u = 1$ the third bound gives at most $N$ translates, while each translate adds at most one dimension, so at least $N$ are needed.
\end{proof}

\begin{remark}\label{rem:doubling}
The logarithm is intrinsic to the multiplicative mechanism and cannot be removed by re-optimizing where one stops.
Suppose we run the growth process only until the codimension drops to $\gamma N$, which costs about $\ln(1/\gamma)/\alpha$ translates of $U$, and then switch to doubling the accumulated space, $W \mapsto W + \rho(g) \cdot W$.
\Cref{thm:first-moment} applied to the pair $(W^\perp, W^\perp)$ squares the relative codimension at each doubling, so about $\log_2\bigl(\ln N / \ln(1/\gamma)\bigr)$ doublings are needed to finish, and since each of them doubles the number of translates in play, they multiply that number by about $\ln N / \ln(1/\gamma)$.
The total is again of order
\[
  \frac{\ln(1/\gamma)}{\alpha} \cdot \frac{\ln N}{\ln(1/\gamma)} = \frac{\ln N}{\alpha},
\]
irrespective of $\gamma$.
What breaks the pattern is a completion step that is structural, and this is exactly what \cite[Proposition 5.2]{JS25} provides.
\end{remark}

\subsection{The complexity of a representation}
The bounds of \Cref{thm:covering} are off from the trivial lower bound $\lceil \dim V / \dim U \rceil$ by a logarithm.
It is convenient to package the discrepancy into one number.

\begin{definition}\label{def:complexity}
Let $H$ act irreducibly on $V$.
For a nonzero subspace $U \leq V$ put
\[
  \rr(V, U) = \frac{\diamp[H](V, U)}{\lceil \dim V / \dim U \rceil},
  \qquad
  \cc(V) = \max_{0 \neq U \leq V} \rr(V, U),
\]
and call $\rr(V, U)$ the \emph{covering ratio} of the pair $(V, U)$ and $\cc(V)$ the \emph{covering complexity} of $V$.
Thus $\rr(V, U) \geq 1$ always, and $\cc(V) = 1$ says that every subspace of $V$ has an additive diameter equal to the trivial bound.
\end{definition}

The acting group is part of the invariant, and we leave it out of the notation because it is usually clear from the context.
Where the same space carries actions of two different groups and the comparison is the point, as in \Cref{subsec:fin-ext,subsec:fin-2trans}, we write $\cc_G(V)$ and display the group.
There is no ambiguity between the two readings of the definition, since $\diamp[G](V, U) = \diamp[K](V, U)$ for a maximal compact subgroup $K \leq G$ by \Cref{lem:compact-algebraic}, and hence both $\rr$ and $\cc$ agree as well.

The invariant is bounded by the covering theorem, and the bound is uniform in the group.

\begin{proposition}\label{prop:cxupper}
Let $H$ act irreducibly on $V$, either as a compact group acting unitarily or as a complex reductive group acting rationally. Then
\[
1 \leq \cc(V) \leq 2 + \ln \dim V.
\]
\end{proposition}

\begin{proof}
The lower bound is the trivial one, and $\cc(V) = 1 \leq 2$ when $\dim V = 1$, so assume $\dim V \geq 2$.
For the upper bound write $N = \dim V$ and $u = \dim U$, and pass to a maximal compact subgroup by \Cref{lem:compact-algebraic} if the acting group is reductive.
Since $\lceil N/u \rceil \geq N/u$, the middle bound of \Cref{thm:covering} gives
\[
  \rr(V, U) \leq \frac{u}{N}\Bigl(\frac{N(1 + \ln u)}{u} + 1\Bigr) = 1 + \ln u + \frac{u}{N} \leq 2 + \ln N. \qedhere
\]
\end{proof}

\subsection{Equivariant morphisms}
\Cref{thm:covering} feeds directly into \cite[Theorem 7.8]{JS25} and bounds the additive diameter $\diamp(V, \im f)$ of an image, the least $m$ for which the $m$-fold sum $\im f + \dots + \im f$ is all of $V$, in terms of a single dimension.
Note that this is the ordinary additive diameter of the subset $\im f$, with no translates involved, in contrast with the group-additive diameter of the subspace $\im(D_w f)$.

\begin{corollary}\label{cor:equivariant}
Let $G$ be a complex reductive linear algebraic group with an irreducible representation on $V$, put $N = \dim V$, and let $f \colon W \to V$ be a $G$-equivariant algebraic morphism.
Then for any $w \in W$ with $D_w f \neq 0$,
\[
  \diamp\bigl(V, \im f\bigr) \leq 2 \frac{N\bigl(1 + \ln \dim \im(D_w f)\bigr)}{\dim \im(D_w f)} + 2.
\]
\end{corollary}

\begin{proof}
Apply \Cref{thm:covering} to $U = \im(D_w f)$ and then \cite[Theorem 7.8]{JS25}, which states that $\diamp(V, \im f) \leq 2\diamp[G](V, \im(D_w f))$.
\end{proof}

In \Cref{subsec:equivariant-witness} we give an example where the bound of the corollary is sharp up to a constant factor, which we found surprising.

\section{Representations of optimal complexity}\label{sec:optimal}

We collect here some examples of representations where $\cc(V) = 1$ and thus every subspace has an optimal additive diameter.
One property implies this and is easier to check, namely that translates meet a given subspace as sparsely as the dimensions allow.
We begin with it, since it is what one would want the intersection bound to say.
Two families follow, the irreducible representations of $\SL_2(\C)$ and the standard representations of the symmetric groups, and the property holds in both.

\subsection{Transversality}\label{subsec:transversal}
Let $N = \dim V$.
Counting dimensions, every pair of subspaces $W, U \leq V$ and every $g$ satisfy
\[
  \dim\bigl(W \cap g \cdot U\bigr) \geq \max\bigl(0, \dim W + \dim U - N\bigr),
\]
since $\dim(W \cap g \cdot U) = \dim W + \dim U - \dim(W + g \cdot U)$ and the sum has dimension at most $N$.
Call $V$ \emph{transversal} if for every pair $W, U \leq V$ some $g$ attains this minimum, so that a general translate of $U$ meets $W$ as sparsely as the dimensions permit.
The name is the one used in intersection theory, where Kleiman's theorem \cite{Kle74} gives exactly this for a group acting transitively on the ambient variety.
Here the group acts non-transitively on $V$, and transversality can genuinely fail, by \Cref{rem:transversality}.

The intersection bound is the weak form of transversality.
Indeed
\[
  \frac{\dim W \cdot \dim U}{N} - \bigl(\dim W + \dim U - N\bigr)
  = \frac{(N - \dim W)(N - \dim U)}{N} \geq 0,
\]
so the right-hand side of \Cref{thm:first-moment} always exceeds the transversal value, with equality only when $W$ or $U$ is $0$ or all of $V$.

\begin{proposition}\label{prop:transversal}
Let $K$ be a compact group with an irreducible unitary representation on $V$.
If $V$ is transversal then $\cc(V) = 1$.
\end{proposition}

\begin{proof}
Let $U \leq V$ be nonzero, put $u = \dim U$ and $m = \lceil N/u \rceil$, and note that $(m-1)u < N \leq mu$.
There is nothing to prove for $m = 1$, so assume $m \geq 2$.

We first build a direct sum of $m-1$ translates.
Suppose $W_j = g_1 \cdot U \oplus \cdots \oplus g_j \cdot U$ has been found with $j \leq m-2$, starting from $W_1 = U$.
Since $\dim W_j + u = (j+1)u \leq (m-1)u < N$, transversality applied to the pair $W_j$ and $U$ produces $g_{j+1}$ with $\dim(W_j \cap g_{j+1} \cdot U) = 0$, and we may set $W_{j+1} = W_j \oplus g_{j+1} \cdot U$.
Iterating gives $W_{m-1}$ of dimension $(m-1)u$.

One more translate finishes, and here we use transversality in the complements.
The subspaces $W_{m-1}^{\perp}$ and $U^{\perp}$ have dimensions summing to $2N - mu \leq N$, so transversality produces $g$ with $W_{m-1}^{\perp} \cap g \cdot U^{\perp} = 0$.
As $\rho$ is unitary, $g \cdot U^{\perp} = (g \cdot U)^{\perp}$, so $(W_{m-1} + g \cdot U)^{\perp} = 0$ and $W_{m-1} + g \cdot U = V$.
Thus $m$ translates of $U$ span $V$, which is the trivial lower bound, and $\rr(V, U) = 1$ for every $U$.
\end{proof}

\begin{example}\label{ex:not-transversal-but-optimal}
The converse of \Cref{prop:transversal} fails, so transversality is strictly stronger than $\cc(V) = 1$.
Here is an example.

Let $H_4 = \langle S, C \rangle$ be the Heisenberg group of \Cref{ex:heisenberg_12} with $12$ replaced by $4$, so that $H_4$ acts irreducibly on $V = \C[\Z/4]$ by $S e_j = e_{j+1}$ and $Ce_j = i^{j}e_j$.
The translates of a coordinate subspace $U_A$ are exactly the $U_{A+a}$ with $a \in \Z/4$.
Transversality fails at the pair
\[
  W = \langle e_1, e_3 \rangle
  \qquad \text{and} \qquad
  U = \langle e_0, e_1 \rangle,
\]
since $\dim W + \dim U = 4$ but $\dim(W \cap g \cdot U) = 1$ for every $g$.

The complexity is nevertheless optimal.
Let now $U \leq V$ be an arbitrary nonzero subspace and put $u = \dim U$.
It suffices to consider the case $u = 2$.
Indeed, $u = 1$ is optimal by \Cref{cor:small}, for $u = 3$ the trivial bound is $2$ and the third bound of \Cref{thm:covering} gives $N - u + 1 = 2$ as well, and $u = 4$ means $U = V$.
So we must produce $g$ with $U \oplus gU = V$.
Pass to Pl\"ucker coordinates.
Writing $\omega \in \Lambda^2 V$ for the Pl\"ucker point of $U$ and $\omega \wedge \omega'$ for the pairing $\Lambda^2 V \times \Lambda^2 V \to \Lambda^4 V \cong \C$, we have $U \cap gU = 0$ if and only if $\omega \wedge g\omega \neq 0$.
So a subspace with no disjoint translate is a nonzero solution of the sixteen quadrics
\[
  \omega \wedge ( S^{s}C^{t} ) \omega = 0,
  \qquad s, t \in \Z/4,
\]
of which the $(0,0)$-equation is the Pl\"ucker relation $\omega \wedge \omega = 0$.
Seven of these already force $\omega = 0$.
Write
\[
  (a, b, c, d, e, f) = (p_{01}, p_{02}, p_{03}, p_{12}, p_{13}, p_{23})
\]
for the Pl\"ucker coordinates of $\omega$ in the basis $e_0, e_1, e_2, e_3$, and index the quadrics by the pairs $(s,t)$.
Up to nonzero scalar factors, the equations for $(0,0)$, $(0,1)$ and $(0,2)$ read
\[
  af - be + cd = 0,
  \qquad
  af - cd = 0,
  \qquad
  af + be + cd = 0 .
\]
The difference and the sum of the outer two give $be = 0$ and $af + cd = 0$, which together with the middle one give $af = be = cd = 0$.
The equations for $(2,0)$ and $(2,1)$ read
\[
  a^2 - c^2 - d^2 + f^2 + 2be = 0,
  \qquad
  a^2 + c^2 + d^2 + f^2 = 0,
\]
so their sum and difference give $a^2 + f^2 = 0$ and $c^2 + d^2 = 0$ once $be = 0$ is used.
Now $af = 0$ makes one of $a, f$ vanish, and then $a^2 = -f^2$ makes the other vanish too. The same argument with $cd = 0$ handles $c$ and $d$.
Thus $a = c = d = f = 0$, and the equations for $(1,0)$ and $(1,1)$ collapse to $-b^2 + e^2 = 0$, $b^2 + e^2 = 0$, whence $b = e = 0$ and $\omega = 0$, a contradiction.
Thus every two-dimensional subspace of $V$ has a disjoint translate and $\diamp[H_4](V, U) = 2$.
Hence $\cc(V) = 1$ while $V$ is not transversal.
\end{example}

\subsection{Exterior powers and transversality}\label{subsec:fin-std}

Transversality can be checked through exterior powers.

\begin{lemma}\label{lem:wedge-transversal}
Let $K$ be a compact group with a unitary representation on $V$ such that $\Lambda^j V$ is irreducible for every $0 < j < N = \dim V$.
Then $V$ is transversal.
\end{lemma}

\begin{proof}
Let $W, U \leq V$ be nonzero and proper, put $w = \dim W$ and $u = \dim U$, and let $\omega_W \in \Lambda^w V$ and $\omega_U \in \Lambda^u V$ be Plücker points, so that $\omega_W \wedge g\omega_U \neq 0$ if and only if $W \cap g \cdot U = 0$.

Suppose first that $w + u \leq N$, so that the minimum in question is $0$.
If $W \cap g \cdot U \neq 0$ for every $g$, then $\omega_W \wedge g\omega_U = 0$ for every $g$.
The vectors $g\omega_U$ span a nonzero invariant subspace of $\Lambda^u V$, hence all of it by irreducibility, so $\omega_W \wedge \zeta = 0$ for every $\zeta \in \Lambda^u V$.
This is false, since in a basis $e_1, \ldots, e_N$ adapted to $W$ we may take $\zeta = e_{w+1} \wedge \cdots \wedge e_{w+u}$, which is legitimate because $w + u \leq N$.
So some translate meets $W$ trivially.

If instead $w + u > N$, apply the previous paragraph to the pair $W^{\perp}$ and $U^{\perp}$, whose dimensions sum to $2N - w - u < N$, using that $\Lambda^{N-u}V$ is irreducible.
It produces $g$ with $W^{\perp} \cap g \cdot U^{\perp} = 0$, and unitarity turns this into $W + g \cdot U = V$, that is $\dim(W \cap g \cdot U) = w + u - N$.
\end{proof}

Here is a concrete case where the criterion applies.

\begin{corollary}\label{cor:sn-std}
Let $n \geq 2$ and let the symmetric group $S_n$ act on the standard representation $V = \C^{n-1}$. Then
\[
  \diamp[S_n](V, U) = \Bigl\lceil \frac{n-1}{\dim U} \Bigr\rceil
  \quad \text{for every nonzero } U \leq V,
  \quad \text{and hence} \quad
  \cc(V) = 1.
\]
\end{corollary}

\begin{proof}
For every $0 \leq j \leq n-1$ the exterior power $\Lambda^j V$ is the irreducible representation of $S_n$ labelled by the hook partition $(n-j, 1^j)$ \cite[Exercise 4.6]{FH91}.
In particular every $\Lambda^j V$ with $0 < j < n-1$ is irreducible, so $V$ is transversal by \Cref{lem:wedge-transversal} and $\cc(V) = 1$ by \Cref{prop:transversal}.
\end{proof}

\subsection{Irreducible representations of \texorpdfstring{$\SL_2(\C)$}{SL(2,C)}}\label{sec:sl2}

Transversality also holds for the irreducible representations of $\SL_2(\C)$.
This gives a second and more conceptual proof of the optimal diameters from \cite[Theorem 3.1]{JS25}.

\begin{theorem}\label{thm:sl2-transversal}
For every $k \geq 1$ the irreducible representation $V_k = \Sym^k \C^2$ of $\SL_2(\C)$ is transversal, and consequently $\cc(V_k) = 1$.
\end{theorem}

\begin{proof}
We work with the concrete model $V_k = \C[X,Y]_k$ of homogeneous polynomials of degree $k$ in two variables, on which $\SL_2(\C)$ acts by linear change of variables.

Put $N = k+1 = \dim V_k$ and let $W, U \leq V_k$ with $w = \dim W$ and $u = \dim U$.
If either subspace is $0$ or all of $V_k$ the required value is forced for every $g$, so assume both are nonzero and proper.

Assume first that $w + u \leq N$, so that the transversal value is $0$.
The locus
\[
  \Sigma_W = \bigl\{[U'] \in \mathrm{Gr}(u, V_k) \bigm| U' \cap W \neq 0 \bigr\}
\]
is closed, so if no translate of $U$ met $W$ trivially, then the whole orbit closure $\overline{\SL_2(\C) \cdot [U]}$ would lie in $\Sigma_W$.
That closure is stable, closed and projective, so the Borel fixed point theorem \cite[Theorem III.10.4]{Bor91} supplies a point of it fixed by the upper triangular Borel subgroup, whose own orbit again lies in $\Sigma_W$.
It is therefore enough to rule out the Borel-stable subspaces, and this is where $\SL_2(\C)$ is special, since there is exactly one of each dimension.
Indeed such a subspace is stable under the diagonal torus, hence spanned by monomials $X^{k-i}Y^i$, and stable under $X\partial_Y$, which lowers $i$, so the one of dimension $u$ is
\[
  F_u = X^{\,k-u+1}\C[X,Y]_{u-1} = \bigl\{ f \in V_k \bigm| X^{\,k-u+1} \text{ divides } f \bigr\}.
\]
Its translates are the subspaces $L^{\,k-u+1}\C[X,Y]_{u-1}$ with $L$ a nonzero linear form, because $\SL_2(\C)$ is transitive on those up to scalars, and they consist of the forms vanishing to order at least $k-u+1$ at the point of $\mathbf{P}^1$ cut out by $L$.

Since $w \leq N - u = k-u+1$, it suffices to produce a point at which no nonzero element of $W$ vanishes to order $w$.
Fix a basis $f_1, \ldots, f_w$ of $W$ and pass to the affine chart, writing $\phi_i(t) = f_i(1,t)$.
This is an isomorphism between $V_k$ and the space of polynomials in $t$ of degree at most $k$, under which vanishing to order $w$ at $[1:t_0]$ becomes divisibility by $(t-t_0)^w$.
The elements of $W$ vanishing to order at least $w$ at $[1:t_0]$ form the kernel of
\[
  W \longrightarrow \C^{w},
  \qquad
  f \longmapsto \bigl(\phi(t_0), \phi'(t_0), \ldots, \phi^{(w-1)}(t_0)\bigr),
\]
whose determinant in the chosen basis is the Wronskian $\mathrm{Wr}(\phi_1, \ldots, \phi_w)(t_0)$.
The $\phi_i$ are linearly independent polynomials, so their Wronskian does not vanish identically, and any $t_0$ off its zero set gives $W \cap L^{\,k-u+1}\C[X,Y]_{u-1} = 0$.

If instead $w + u > N$, pass to annihilators in $V_k^*$.
We have $W + g \cdot U = V_k$ if and only if $W^\perp \cap (g \cdot U)^\perp = 0$, the annihilator $(g\cdot U)^\perp$ is the translate of $U^\perp$ under the dual action, and $V_k^* \cong V_k$ as a module.
As $\dim W^\perp + \dim U^\perp = 2N - w - u < N$, the case already settled applies to the pair $(W^\perp, U^\perp)$ and returns $g$ with $W + g \cdot U = V_k$, that is $\dim(W \cap g \cdot U) = w+u-N$.

In both cases the admissible $g$ form a nonempty Zariski open subset of $\SL_2(\C)$, since the dimension in question is the rank of a matrix depending regularly on $g$, so the Zariski dense subgroup $\SU(2)$ contains such a $g$ as well.
Hence $V_k$ is transversal over $\SL_2(\C)$ and over $\SU(2)$, and \Cref{prop:transversal} gives $\cc(V_k) = 1$.
\end{proof}

\begin{remark}
The representations $V_k$ of $\SL_2(\C)$ are transversal by \Cref{thm:sl2-transversal} while $\Lambda^d V_k$ is far from irreducible, so \Cref{lem:wedge-transversal} is not a necessary condition for transversality.
\end{remark}

\section{Large complexity from structured subspaces}\label{sec:degree}

Our aim in this section is to show that $\ln \dim V$ is the right order.
We construct representations $V$ with particularly structured subspaces $U$ for which $\rr(V, U)$ is of order $\ln \dim V$.

\subsection{Principal subspaces in a Cartan section ring}
Let $G$ be a connected reductive group over $\C$, fix a Borel subgroup and a maximal torus, let $\lambda$ be a nonzero dominant weight and let $P_\lambda$ be the stabilizer of the highest weight line in $\mathbf{P}(V(\lambda))$.
Put
\[
  X_\lambda = G/P_\lambda \subseteq \mathbf{P}(V(\lambda)),
  \qquad
  D = \dim X_\lambda,
\]
so that $X_\lambda$ is the closed orbit and carries the ample homogeneous line bundle $\mathcal{L}_\lambda$ induced by the embedding. A general reference for the following is \cite{Bri05}.

\begin{lemma}\label{lem:cartan-ring}
There are $G$-equivariant identifications $H^0(X_\lambda, \mathcal{L}_\lambda^{\otimes m}) \cong V(m\lambda)^*$ for $m \geq 0$, under which multiplication of sections is the Cartan product
\[
  V(a\lambda)^* \otimes V(b\lambda)^* \longrightarrow V\bigl((a+b)\lambda\bigr)^*,
\]
and this map is surjective.
Consequently
\[
  R(\lambda) = \bigoplus_{m \geq 0} V(m\lambda)^*
\]
is a graded integral domain generated in degree one, of Krull dimension $D+1$, whose Hilbert function $h(m) = \dim V(m\lambda)$ agrees for large $m$ with a polynomial of degree $D$ and positive leading coefficient.
\end{lemma}

\begin{proof}
The identification of the spaces of sections is the Borel--Weil theorem.
Since $X_\lambda$ is irreducible, the section ring is a domain, because on a common trivializing dense open set nonzero sections are nonzero regular functions and their product cannot vanish identically.
In particular the multiplication map is nonzero, and being $G$-equivariant with irreducible target $V((a+b)\lambda)^*$, it is surjective.
Taking $a = 1$ repeatedly gives generation in degree one.

Finally the section ring of an ample line bundle on a projective variety of dimension $D$ has Krull dimension $D+1$, and its Hilbert polynomial has degree $D$ and positive leading coefficient.
\end{proof}

For the rest of the subsection write $R = R(\lambda)$ and $R_m = V(m\lambda)^*$.
Fix $k \geq 1$, put
\[
  V_k = V(k\lambda) = R_k^*,
\]
choose a nonzero $\ell \in R_1$, and define
\begin{equation}\label{eq:principal}
  U_{\ell,k} = (\ell R_{k-1})^{\perp} \leq R_k^* = V_k,
\end{equation}
the annihilator being taken for the natural perfect pairing between $R_k^*$ and $R_k$.
Two elementary facts are used repeatedly.
Equivariance of the pairing gives
\begin{equation}\label{eq:equivariance}
  g U_{\ell,k} = U_{g\ell,k},
  \qquad
  \bigl(g U_{\ell,k}\bigr)^{\perp} = (g\ell)R_{k-1},
\end{equation}
so that the translates of $U_{\ell,k}$ are exactly the $U_{g\ell,k}$, and for subspaces $A_i \leq V_k$ we have
\begin{equation}\label{eq:annsum}
  \Bigl(\sum_i A_i\Bigr)^{\perp} = \bigcap_i A_i^{\perp}.
\end{equation}

\begin{theorem}\label{thm:cartan}
Assume $D = \dim G/P_\lambda \geq 1$. Then $U_{\ell,k} \neq 0$,
\[
  \dim U_{\ell,k} = \dim R_k - \dim R_{k-1},
  \qquad
  \diamp[G](V_k, U_{\ell,k}) \geq k+1,
\]
and moreover
\[
  \liminf_{k \to \infty} \rr(V_k, U_{\ell,k}) \geq D.
\]
\end{theorem}

\begin{proof}
As $R$ is a domain and $\ell \neq 0$, multiplication by $\ell$ is an injection $R_{k-1} \hookrightarrow R_k$, so $\dim \ell R_{k-1} = \dim R_{k-1}$ and the dimension of $U_{\ell,k}$ is as claimed.
By the Weyl dimension formula $\dim R_k > \dim R_{k-1}$, so $U_{\ell,k}$ is nonzero.

Let $m \leq k$ and let $g_1, \ldots, g_m \in G$.
Take any nonzero $r \in R_{k-m}$, which exists because that graded piece contains the $(k-m)$-th Cartan power of any nonzero element of $R_1$.
The product
\[
  q = (g_1\ell)(g_2\ell)\cdots(g_m\ell)\,r \in R_k
\]
is nonzero because $R$ is a domain, and for each $i$ it lies in the principal subspace $(g_i\ell)R_{k-1}$.
By \eqref{eq:equivariance} and \eqref{eq:annsum},
\[
  0 \neq q \in \bigcap_{i \leq m}(g_i\ell)R_{k-1} = \Bigl(\sum_{i \leq m} g_i U_{\ell,k}\Bigr)^{\perp},
\]
so $m$ translates never span $V_k$ and the diameter is at least $k+1$.

It remains to compare this with the trivial bound.
Write the leading term of the Hilbert polynomial as $ct^D$ with $c > 0$, so that
\[
  N_k = \dim V_k = ck^D + O(k^{D-1}),
\]
and the dimension formula above turns the Hilbert function into its difference,
\[
  u_k = \dim U_{\ell,k} = h(k) - h(k-1) = cDk^{D-1} + O(k^{D-2}).
\]
Using $\lceil N_k/u_k \rceil \leq N_k/u_k + 1$ we get
\[
  \rr(V_k, U_{\ell,k}) \geq \frac{k+1}{\lceil N_k/u_k \rceil} \geq \frac{(k+1)u_k}{N_k + u_k} = D + o(1),
\]
which is the assertion about the lower limit, as desired.
\end{proof}

The statement is a lower limit and not a limit, since the true diameter may well exceed $k+1$.
The hypothesis $D \geq 1$ is needed and not merely convenient.

\begin{remark}\label{rem:zerodim}
If $D = 0$ then $P_\lambda = G$, every $V(k\lambda)$ is one-dimensional, multiplication by $\ell$ is an isomorphism $R_{k-1} \to R_k$, and $U_{\ell,k} = 0$, which \Cref{def:complexity} excludes.
The extreme case is a torus, all of whose irreducible representations are one-dimensional and have covering complexity one.
\end{remark}

Taking $\lambda$ regular makes $P_\lambda$ a Borel subgroup and gives the largest value of $D$, so the invariant $\cc$ of a fixed group is bounded below by the dimension of its flag variety.

\begin{corollary}\label{cor:flagbound}
Let $G$ be connected reductive with $\dim G/B > 0$ and let $\lambda$ be a regular dominant weight. Then
\[
  \liminf_{k \to \infty} \cc\bigl(V(k\lambda)\bigr) \geq \dim G/B = |\Phi^+|.
\]
\end{corollary}

\begin{proof}
For a regular dominant weight $P_\lambda = B$, so \Cref{thm:cartan} applies with $D = \dim G/B$, and $\cc(V(k\lambda)) \geq \rr(V(k\lambda), U_{\ell,k})$.
\end{proof}

For $\SL_2(\C)$ this says $\cc \geq 1$, which is no information, and indeed \Cref{sec:sl2} records that $\cc(V) = 1$ there for every $V$.
The first case with content is $\SL_3(\C)$, where the flag variety is a threefold.

\begin{example}\label{ex:sl3flag}
Let $G = \SL_3(\C)$ and $\lambda = \omega_1 + \omega_2$, so that $G/B$ is the full flag variety of dimension three and $R_1 = V(\lambda)^* \cong \slalg_3(\C)$.
By the Weyl dimension formula $\dim V(k,k) = (k+1)^3$, so for $V_k = V(k,k)$ and a nonzero $\ell \in R_1$ the subspace $U_{\ell,k}$ of \Cref{thm:cartan} has
\[
  N_k = (k+1)^3, \qquad u_k = (k+1)^3 - k^3 = 3k^2 + 3k + 1,
  \qquad \diamp[G](V_k, U_{\ell,k}) \geq k+1,
\]
and therefore
\[
  \rr(V_k, U_{\ell,k}) \geq \frac{k+1}{\bigl\lceil (k+1)^3/(3k^2 + 3k + 1) \bigr\rceil}
  \longrightarrow 3.
\]
\end{example}

\subsection{Symmetric powers}
For $\lambda = \omega_1$ the construction is classical.
There $X_\lambda = \mathbf{P}^{n-1}$, the ring $R$ is a polynomial ring, and the subspaces $U_{\ell,k}$ are the symmetric powers of hyperplanes.
It is worth recording the resulting statement in its own terms, since the combinatorics is transparent and it is what we iterate below.

\begin{proposition}\label{prop:sym}
Let $n \geq 2$, let $E = \C^n$, let $k \geq 0$ and let $X \leq E$ be a hyperplane. Then
\[
  \diamp[\SL_n(\C)]\bigl(\Sym^k E, \Sym^k X\bigr) = k+1,
  \qquad
  \rr = \frac{k+1}{1 + \lceil k/(n-1) \rceil}.
\]
\end{proposition}
\begin{proof}
Let $V = \Sym^k E$, and identify $V^*$ with the space $\C[x_1, \ldots, x_n]_k$ of forms of degree $k$ on $E$.
Polarization writes any product of $k$ vectors as a combination of $k$-th powers, and the pairing satisfies $\langle f, w^k \rangle = k!\,f(w)$, so for a subspace $X \leq E$ we get
\[
  (\Sym^k X)^{\perp} = \{f \mid f|_X = 0\} = I(X)_k,
\]
the degree $k$ part of the homogeneous ideal of $X$.
Since $\bigcap_i I(X_i) = I(\bigcup_i X_i)$, a sum $\sum_{i \leq m}\Sym^k X_i$ is all of $\Sym^k E$ if and only if no nonzero form of degree $k$ vanishes on $X_1 \cup \cdots \cup X_m$.

Now take $X$ to be a hyperplane $\ker \lambda$.
We have $I(X)_k = \lambda \cdot \C[x]_{k-1}$, and for pairwise non-proportional $\lambda_1, \ldots, \lambda_m$ the intersection of the ideals is $(\lambda_1 \cdots \lambda_m)$, whose degree $k$ part is nonzero exactly when $m \leq k$.
So at least $k+1$ translates are needed, and $k+1$ pairwise non-proportional forms do the job, all of the corresponding hyperplanes lying in the single $\SL_n$-orbit of $X$.
This proves the diameter formula.
The formula for $\rr$ follows because $N/u = \binom{n+k-1}{k}/\binom{n+k-2}{k} = (n+k-1)/(n-1)$.
\end{proof}

The ratio in \Cref{prop:sym} equals its largest value $(k+1)/2$ for every $n \geq k+1$, and among these the balanced member $n = k+1$ is the one of smallest dimension.

\subsection{Exterior powers}
Another family, of a different shape, reaches the same order.

\begin{lemma}\label{lem:wedge-lower}
Let $E = \C^n$, let $1 \leq j \leq n-1$, and let $W_1, \ldots, W_m \leq E$ be hyperplanes with $m \leq j$. Then
\[
  \sum_{i \leq m} \Lambda^j W_i \neq \Lambda^j E.
\]
\end{lemma}

\begin{proof}
Write $W_i = \ker \varphi_i$ for nonzero functionals $\varphi_i \in E^*$, so that $(\Lambda^j W_i)^\perp = \varphi_i \wedge \Lambda^{j-1}E^*$.
Let $r$ be the dimension of $\langle \varphi_1, \ldots, \varphi_m \rangle$, so that $r \leq m \leq j$, and put $\psi = \varphi_{i_1} \wedge \cdots \wedge \varphi_{i_r} \neq 0$ for a basis of that span.
Every nonzero $\varphi_i$ lies in that span and can be completed to a basis of it, so $\psi \in \varphi_i \wedge \Lambda^{r-1}E^*$ for each $i$.
Since $r \leq j \leq n-1$, there is a $(j-r)$-form $\eta$ on a complementary subspace with $\psi \wedge \eta \neq 0$, and this form lies in every $\varphi_i \wedge \Lambda^{j-1}E^*$.
Hence
\[
  0 \neq \psi \wedge \eta \in \bigcap_{i \leq m} \bigl(\Lambda^j W_i\bigr)^\perp
  = \Bigl(\sum_{i \leq m} \Lambda^j W_i\Bigr)^\perp,
\]
so the sum is proper.
\end{proof}

\begin{remark}
For any subgroup $K \leq \GL(E)$ the translates of $\Lambda^j W$ are again exterior powers of hyperplanes, so \Cref{lem:wedge-lower} gives $\diamp[K](\Lambda^j E, \Lambda^j W) \geq j+1$ for every such $K$ at once, provided only that $\Lambda^j E$ is irreducible for $K$.
Shrinking the group only shrinks the orbit of $[U]$, so the obstruction survives every restriction.
\end{remark}

For the full special linear group the lower bound is exact.

\begin{proposition}\label{prop:ext}
Let $E = \C^n$, let $1 \leq j \leq n-1$ and let $W \leq E$ be a hyperplane. Then
\[
  \diamp[\SL_n(\C)]\bigl(\Lambda^j E, \Lambda^j W\bigr) = j+1,
  \qquad
  \rr = \frac{j+1}{\lceil n/(n-j) \rceil}.
\]
\end{proposition}

\begin{proof}
The lower bound $\diamp[\SL_n(\C)] \geq j+1$ is \Cref{lem:wedge-lower} applied to the hyperplanes $gW$.
Write $W = \ker \varphi$ and choose a basis $e_1, \ldots, e_n$ of $E$ with $\varphi = x_n$, so that $\Lambda^j W$ is spanned by the monomials $e_I$ with $n \notin I$.
Its annihilator is therefore spanned by the $x_I$ with $n \in I$, that is
\[
  (\Lambda^j W)^{\perp} = \varphi \wedge \Lambda^{j-1}E^*,
  \qquad \dim \Lambda^j W = \binom{n-1}{j}.
\]
For the upper bound, choose independent functionals $\varphi_1, \ldots, \varphi_{j+1}$, which is possible because $j+1 \leq n$, and extend them to a basis $x_1, \ldots, x_n$ of $E^*$ with $\varphi_i = x_i$.
A wedge monomial lies in $x_i \wedge \Lambda^{j-1}E^*$ exactly when it contains the letter $x_i$, and no monomial of degree $j$ can contain all of $x_1, \ldots, x_{j+1}$, so
\[
  \bigcap_{i \leq j+1} \varphi_i \wedge \Lambda^{j-1}E^* = 0
\]
and $j+1$ translates span $\Lambda^j E$.
The formula for $\rr$ follows from $N/u = \binom{n}{j}/\binom{n-1}{j} = n/(n-j)$.
\end{proof}

\subsection{Families with a logarithmic complexity}

\begin{corollary}\label{cor:sharp}
Let $V = \Sym^k \C^{k+1}$ as a representation of $G = \SL_{k+1}(\C)$, or $V = \Lambda^m \C^{2m}$ as a representation of $G = \SL_{2m}(\C)$.
In both cases $V$ has a subspace $U$ with $\dim U = \dim V/2$ and
\[
  \diamp[G](V, U) = \frac{\ln \dim V}{2\ln 2}\bigl(1 + o(1)\bigr),
  \qquad
  \cc(V) \geq \rr(V, U) \geq \frac{\ln \dim V}{4 \ln 2}.
\]
\end{corollary}

\begin{proof}
In the first case $N = \binom{2k}{k}$ and $u = \binom{2k-1}{k} = N/2$, so the trivial bound is $2$, and \Cref{prop:sym} gives $\diamp[G](V, U) = k+1$ and $\rr = (k+1)/2$.
In the second case $N = \binom{2m}{m}$ and $u = \binom{2m-1}{m} = N/2$, and \Cref{prop:ext} gives $\diamp[G](V, U) = m+1$ and $\rr = (m+1)/2$.
The two estimates now come from different halves of the same binomial coefficient.
Stirling gives $\ln \binom{2r}{r} = 2r\ln 2 - (\ln(\pi r))/2 + o(1)$ as $r \to \infty$, which is the asymptotic formula for the diameter.
The inequality for $\rr$ needs no asymptotics at all, since the crude bound $\binom{2r}{r} < 4^r$ already gives $\ln N/(4\ln 2) < r/2 \leq (r+1)/2 = \rr$.
\end{proof}

\begin{remark}\label{rem:transversality}
The pair $U = W = \Sym^k \C^k$ inside $V = \Sym^k \C^{k+1}$ also shows how far from transversal the intersections can be.
Here $\dim U = \dim W = N/2$, so transversality would predict $\dim(W \cap \rho(g) \cdot U) = 0$, whereas $\Sym^k X \cap \Sym^k X' = \Sym^k(X \cap X')$ gives
\[
  \dim\bigl(W \cap \rho(g) \cdot U\bigr) = \binom{2k-2}{k}
  = \frac{k-1}{2(2k-1)}N \longrightarrow \frac{N}{4}
\]
for every $g$ outside the stabilizer.
The first moment bound of \Cref{thm:first-moment} is $\dim U \dim W/N = N/4$, so it is attained here asymptotically.
The deficiency from transversality is here a constant fraction of $N$ while $N/\dim U = 2$, so no bound on the deficiency in terms of $N/\dim U$ can hold in general.
\end{remark}

\subsection{The obstruction ignores the group}\label{subsec:fin-ext}

The lower bound of \Cref{lem:wedge-lower} never looks at the group, so it survives restriction to any subgroup for which the exterior power stays irreducible.
Restricting to the symmetric groups makes the point sharply, because the same group appeared in \Cref{sec:optimal} with the smallest complexity the invariant allows.

\begin{proposition}\label{prop:sn-ext}
Let $m \geq 2$, let the symmetric group $S_{2m+1}$ act on the standard representation $E = \C^{2m}$, and put $V = \Lambda^m E$.
Then $V$ is irreducible, and for $U = \Lambda^m W$ with $W \leq E$ a hyperplane,
\[
  \dim U = \dim V/2,
  \quad
  \diamp[S_{2m+1}](V, U) \geq m+1,
  \quad
  \rr(V, U) \geq \frac{m+1}{2} > \frac{\ln \dim V}{4 \ln 2},
\]
and consequently $\cc(V) \geq \ln \dim V/(4 \ln 2)$.
\end{proposition}

\begin{proof}
The pair $(V, U)$ is exactly the exterior witness of \Cref{cor:sharp}, namely $\Lambda^m \C^{2m}$ together with the exterior power of a hyperplane, and only the group has changed, from $\SL_{2m}(\C)$ to a symmetric group acting through its standard representation.
The module $V = \Lambda^m E$ is the hook $(m+1, 1^m)$ as in the proof of \Cref{cor:sn-std}, hence irreducible.
Every translate $gU$ with $g \in S_{2m+1}$ equals $\Lambda^m(gW)$ for the hyperplane $gW \leq E$, so \Cref{lem:wedge-lower} with $j = m$ shows that no $m$ translates of $U$ span $V$, and $\diamp[S_{2m+1}](V, U) \geq m+1$.

As for the numbers, $N = \dim V = \binom{2m}{m}$ and $\dim U = \binom{2m-1}{m} = N/2$, so the trivial bound is $\lceil N/\dim U \rceil = 2$ and $\rr(V, U) \geq (m+1)/2$.
Finally $N \leq 4^m$ gives $\ln N \leq 2m \ln 2$, and hence $(m+1)/2 > m/2 \geq \ln N/(4 \ln 2)$.
\end{proof}

\begin{remark}\label{rem:alternating}
The proposition restricts once more, to the alternating groups.
Recall that an irreducible $S_n$-module restricts irreducibly to $A_n$ unless its partition is self-conjugate, in which case it splits into two nonisomorphic irreducibles \cite[Section 5.1]{FH91}.
The middle hook $(m+1, 1^m)$ is self-conjugate, so $\Lambda^m \C^{2m}$ splits upon restriction to $A_{2m+1}$, but the hook $(m+2, 1^{m-1})$ is not, so $V = \Lambda^{m-1} \C^{2m}$ stays irreducible for $A_{2m+1}$.
For a hyperplane $W \leq \C^{2m}$ and $m \geq 2$, \Cref{lem:wedge-lower} with $j = m-1$ then gives $\diamp[A_{2m+1}](V, \Lambda^{m-1} W) \geq m$ against a trivial bound of two, and hence $\cc(V) \geq m/2 \geq \ln \dim V/(4 \ln 2)$ exactly as before.

Even degrees behave the same way.
For $n = 2m$ with $m \geq 2$ the standard module of $S_{2m}$ is $E = \C^{2m-1}$, and no hook of $S_{2m}$ is self-conjugate, since $(2m-j, 1^j)$ conjugates to $(j+1, 1^{2m-j-1})$ and $2m - j = j+1$ has no integer solution.
So $V = \Lambda^{m-1} E$ is irreducible for $S_{2m}$ and for $A_{2m}$ at once.
\Cref{lem:wedge-lower} with $j = m-1$ gives $\diamp[A_{2m}](V, \Lambda^{m-1}W) \geq m$ for a hyperplane $W \leq E$, and hence the same over $S_{2m}$.
The trivial bound is again two because $\dim V/\dim \Lambda^{m-1}W = (2m-1)/m$.
Since $\dim V = \binom{2m-1}{m-1} \leq 4^m$, we get $\cc(V) \geq m/2 \geq \ln \dim V/(4 \ln 2)$ along the even degrees as well.
So every symmetric and alternating group of either parity carries an exterior power of its standard module on which $\cc$ is logarithmic in the dimension.
\end{remark}

The above should be compared with \Cref{cor:sn-std}.
Both statements are about symmetric groups, and the module $\Lambda^m \C^{2m}$ of \Cref{prop:sn-ext} is built from the standard module of \Cref{cor:sn-std}, yet
\[
  \cc_{S_n}\bigl(\C^{n-1}\bigr) = 1
  \qquad \text{while} \qquad
  \cc_{S_{2m+1}}\bigl(\Lambda^m \C^{2m}\bigr) \geq \frac{\ln \binom{2m}{m}}{4 \ln 2}.
\]
So the complexity moves from the smallest value the invariant takes to within a constant of the largest one that \Cref{prop:cxupper} allows.

\subsection{Equivariant morphisms}\label{subsec:equivariant-witness}

We now apply the previous constructions to show that the bound of \Cref{cor:equivariant} for the diameter with respect to the image of an equivariant morphism is sharp up to a constant factor.

Let $G = \SL_{k+1}(\C)$, $E = \C^{k+1}$, $V = \Sym^k E$ and $N = \dim V = \binom{2k}{k}$.
We write
\[
  Z = \bigcup_{X \leq E \text{ a hyperplane}} \Sym^k X \ \subseteq V
\]
for the forms supported on some hyperplane.
This union is an algebraic variety, as the following incidence construction shows.
Consider
\[
  T = \bigl\{ ([\lambda], z) \in \mathbf{P}(E^*) \times V \ \bigm| \ z \in \Sym^k(\ker \lambda) \bigr\} .
\]
The fibre of $T$ over a point $[\lambda]$ is the linear subspace $\Sym^k(\ker \lambda) \leq V$ of dimension $\binom{2k-1}{k} = N/2$, and $T$ is the total space of a vector bundle of rank $N/2$ over $\mathbf{P}(E^*) \cong \mathbf{P}^k$.
In particular $T$ is irreducible of dimension $N/2 + k$.
The projection $\mathbf{P}(E^*) \times V \to V$ is proper and carries $T$ onto $Z$, and hence $Z$ is a closed irreducible subvariety of $V$ with
\[
  \dim Z \leq \frac{N}{2} + k .
\]

Fix $r \geq N/2$, put $W = E^* \oplus E^{\oplus 2r}$ and define $f \colon W \to V$ by
\[
  f(\lambda, u_1, u_1', \ldots, u_r, u_r')
  = \sum_{i \leq r} \bigl( \lambda(u_i') u_i - \lambda(u_i) u_i' \bigr)^{k}.
\]
Writing $v_i = \lambda(u_i')u_i - \lambda(u_i)u_i'$ we have $\lambda(v_i) = 0$, so each $v_i$ lies in the hyperplane $\ker \lambda$ and the value of $f$ lies in $Z$.
In words, the first argument chooses a hyperplane and the remaining ones choose $r$ vectors inside it.

\begin{proposition}\label{prop:equivariant-witness}
Let $k \geq 4$.
Then $f$ is a $G$-equivariant algebraic morphism with $\im f = Z$, and for a general $w \in W$ we have $N/2 \leq \dim \im(D_w f) \leq N/2 + k$ and
\[
  k - 2 \leq \diamp[G](V, \im(D_w f)) \leq k+1,
  \qquad
  k - 2 \leq \diamp\bigl(V, \im f\bigr) \leq 2k+2
\]
with
\[
  \rr(V, \im(D_w f)) > \frac{\ln \dim V}{4 \ln 2} - 1.
\]
\end{proposition}

\begin{proof}
Let $g \in G$ act on $E^*$ by $(g\lambda)(u) = \lambda(g^{-1}u)$.
Then $(g\lambda)(gu_i') \cdot gu_i - (g\lambda)(gu_i) \cdot gu_i' = g v_i$, and hence $f(g \cdot w) = g \cdot f(w)$, so $f$ is equivariant.
For the image, fix $\lambda$ and a vector $u'$ with $\lambda(u') = 1$.
As $u$ ranges over $E$, the vector $u - \lambda(u)u'$ ranges over all of $X = \ker \lambda$, so the values of $f$ with $\lambda$ fixed are the sums of $r$ $k$-th powers of elements of $X$.
These powers span $\Sym^k X$, and every scalar in $\C$ has a $k$-th root, so any $\sum_j c_j w_j^{k}$ is already a sum of $k$-th powers.
Since $r \geq N/2 = \dim \Sym^k X$, we get $\im f = Z$.

Let $U = \im(D_w f)$.
We now bound $\dim U$ from both sides.
The map obtained from $f$ by fixing $\lambda$ is a dominant morphism onto $\Sym^k X$, so its differential at a general point is surjective and $U \supseteq \Sym^k X$, whence $\dim U \geq \dim \Sym^k X = N/2$.
In the other direction, $f$ is a dominant morphism from the smooth variety $W$ onto the irreducible variety $Z$ built above, so generic smoothness \cite[III, Corollary 10.7 and Proposition 10.4]{Har77} makes $f$ smooth over a dense open subset of the smooth locus of $Z$.
For a general $w$ the differential $D_w f$ is therefore surjective onto the tangent space of $Z$ at $f(w)$, and hence
\[
  \dim U = \dim Z \leq \frac{N}{2} + k .
\]
Since $k \geq 4$ we have $k < \binom{2k-1}{k} = N/2$, so $N/2 + k < N$ and the two bounds read $N/2 \leq \dim U < N$.
Consequently the trivial bound is $\lceil N / \dim U \rceil = 2$.

Everything else rests on one count, and here repeated hyperplanes have to be allowed.
Identify $V^*$ with $\C[x_1, \ldots, x_{k+1}]_k$ as in \Cref{prop:sym}, so that the annihilator of $\Sym^k X$ is the space $\lambda \C[x]_{k-1}$ of forms divisible by $\lambda$, where $X = \ker \lambda$.
Let $X_1, \ldots, X_m$ be any hyperplanes with $m \leq k$, cut out by linear forms $\lambda_1, \ldots, \lambda_m$, and let $\lambda_1, \ldots, \lambda_r$ be the distinct ones among them after renumbering.
A form annihilating every $\Sym^k X_i$ is divisible by each of the $r$ pairwise non-proportional forms $\lambda_1, \ldots, \lambda_r$, hence by their product, and therefore
\[
  \bigcap_{i \leq m} (\Sym^k X_i)^{\perp}
  = (\lambda_1 \cdots \lambda_r) \C[x]_{k-r},
  \quad \text{of dimension} \quad
  \binom{2k-r}{k} \geq \binom{2k-m}{k} .
\]
The inequality holds because $r \leq m$, so the count $\binom{2k-m}{k}$ is a lower bound for every $m$-tuple of hyperplanes, repetitions included.
We claim that
\begin{equation}\label{eq:witness-count}
  \binom{2k-m}{k} > mk
\end{equation}
forces both that $m$ translates of $U$ fail to span $V$ and that $mZ \neq V$.
The quantity $mk$ is the number of parameters needed to move the $m$ hyperplanes, and \eqref{eq:witness-count} says that the forms surviving all $m$ of them outnumber those parameters.

For the first assertion, take $g_1, \ldots, g_m \in G$ and put $X_i = g_i X$, a hyperplane for each $i$ and possibly with repetitions.
By the dimension bound above, $U^{\perp}$ sits inside $(\Sym^k X)^{\perp}$ with codimension at most $k$, so each $(g_i U)^{\perp}$ is cut out of $(\Sym^k X_i)^{\perp}$ by at most $k$ linear conditions.
Imposing these $mk$ conditions on the intersection just computed leaves
\[
  \dim \bigcap_{i \leq m} (g_i U)^{\perp}
  \geq \binom{2k-m}{k} - mk > 0 ,
\]
so the translates $g_1 U, \ldots, g_m U$ do not span $V$.
For the second, every point of $mZ$ lies in a sum $\sum_{i \leq m} \Sym^k X_i$ for some $m$-tuple of hyperplanes, and the count above bounds the codimension of such a sum below by $\binom{2k-m}{k}$, again with repetitions allowed.
These sums form a family of subspaces over the $mk$-dimensional parameter space $(\mathbf{P}(E^*))^m$, and hence
\[
  \dim mZ \leq N - \binom{2k-m}{k} + mk < N .
\]

Take $m = k-3$.
This turns \eqref{eq:witness-count} into $\binom{k+3}{3} > (k-3)k$, which holds for every $k \geq 4$.
Thus $\diamp[G](V,U) \geq k-2$ and $\diamp(V, \im f) \geq k-2$.
The upper bound on the left is \Cref{prop:sym} applied to $\Sym^k X \subseteq U$, and the upper bound on the right is \cite[Theorem 7.8]{JS25} applied to the left.
Finally $N < 4^k$ gives $k > \ln N / (2 \ln 2)$, and the estimate for $\rr(V,U)$ follows, as desired.
\end{proof}

\section{Large complexity from small groups}\label{sec:small}

Finite groups are compact, so \Cref{thm:covering} and \Cref{def:complexity} apply to them, and since a fixed finite group has only finitely many irreducible representations, every statement in this section is about families.
The representations below make \Cref{prop:cxupper} sharp in the leading constant, and together with the covering theorem they complete \Cref{thm:C} of the introduction, which pins down the largest possible order of the complexity.

\subsection{The Heisenberg group and additive covering}\label{subsec:fin-heis}

We first recall the Heisenberg group $H_p$ of order $p^3$ for an odd prime $p$.
Let $\omega = e^{2\pi i/p}$, and let the operators $X$ and $Z$ act on $V = \C[\Z/p]$, with basis $(e_j)_{j \in \Z/p}$, by
\[
  X e_j = e_{j+1}
  \qquad \text{and} \qquad
  Z e_j = \omega^j e_j.
\]
The Heisenberg group is $H_p = \langle X, Z \rangle \leq \GL(V)$.
We have $ZXZ^{-1}X^{-1} = \omega \Id_V$, and every element of $H_p$ has the form $\omega^c X^a Z^b$ with $a, b, c \in \Z/p$.

The $p^2$ operators $X^a Z^b$ are linearly independent, since operators with different $a$ occupy disjoint cyclic diagonals, while for a fixed $a$ the diagonal coefficient vectors $(\omega^{bj})_{j \in \Z/p}$ form the Fourier basis of $\C^p$ as $b$ varies.
As $\dim \End(V) = p^2$, they are a basis of $\End(V)$.
Every $H_p$-invariant subspace of $V$ is therefore invariant under all of $\End(V)$ and so is $0$ or $V$, which makes $V$ irreducible, the \emph{Schrödinger representation} of $H_p$.
For a nonempty subset $S \subseteq \Z/p$ let $U_S = \langle e_s \mid s \in S \rangle$ be the coordinate subspace on $S$.

\begin{proposition}\label{prop:heis-dict}
For every nonempty $S \subseteq \Z/p$,
\[
  \diamp[H_p]\bigl(\C[\Z/p], U_S\bigr) = \min\bigl\{|T| \bigm| T \subseteq \Z/p,\ S + T = \Z/p\bigr\},
\]
the least number of translates of $S$ needed to cover $\Z/p$, against the trivial bound $\lceil p/|S| \rceil$.
\end{proposition}

\begin{proof}
Every element of $H_p$ has the form $\omega^c X^a Z^b$.
Scalars act trivially on subspaces, the modulation $Z^b$ is diagonal and therefore fixes every coordinate subspace, and $X^a U_S = U_{S+a}$.
So the translates of $U_S$ are exactly the subspaces $U_{S+t}$ with $t \in \Z/p$.
A sum of coordinate subspaces is the coordinate subspace on the union of the index sets, so $\sum_{t \in T} U_{S+t} = U_{S+T}$, and this is all of $V$ if and only if $S + T = \Z/p$.
Repeated translates never help, so minimizing the number of group translates is the same as minimizing $|T|$.
\end{proof}

The right-hand side of the proposition is the \emph{additive covering number} of $S$, denoted by $\tau(S, \Z/p)$.
This quantity has been studied since Lorentz \cite{Lor54} and Newman \cite{New67}, and in the generality of arbitrary groups by Bollobás, Janson and Riordan \cite{BJR11}.
That paper writes
\[
  \kappa(S,G) = \frac{\tau(S,G)}{|G|/|S|}
\]
for the \emph{covering multiplicity}, which is the covering number normalized by the trivial bound.
In this notation \Cref{prop:heis-dict} says that $\rr(V, U_S)$ differs from $\kappa(S, \Z/p)$ only through the ceiling in the trivial dimension bound, and the lower bounds proved for $\kappa$ transfer directly.

\begin{theorem}\label{thm:heis}
Let $\epsilon > 0$.
For every sufficiently large prime $p$ there is a subset $S \subseteq \Z/p$ with
\[
  \rr\bigl(\C[\Z/p], U_S\bigr) \geq (1 - \epsilon) \ln p.
\]
Consequently, $\cc(\C[\Z/p]) = (1 + o(1)) \ln \dim \C[\Z/p]$ as $p \to \infty$.
\end{theorem}

\begin{proof}
We may assume $\epsilon < 1$, and we put $\delta = \epsilon/2$ and $k = \lceil p/\ln p \rceil$.
By \cite[Theorem 4.7]{BJR11}, if $n = n(k)$ and $h = h(k)$ satisfy $n \geq h$, $h/k \to \infty$ and $n \leq k^{1 + \delta/7}$, then for any group $G$ of order $n$ and any $H \subseteq G$ with $|H| = h$, a proportion $1 - o(1)$ of the $k$-element subsets $S \subseteq H$ satisfy $\kappa(S, G) \geq (1-\delta)\ln k$.
We apply this with $G = H = \Z/p$, so that $n = h = p$.
The hypotheses hold, since $p/k = (1 + o(1))\ln p \to \infty$, while $p \leq k^{1+\delta/7}$ is the same as $p/k \leq k^{\delta/7}$, and the left-hand side is asymptotic to $\ln p$ against a fixed positive power of $p$ on the right.
So for all large $p$ there is a $k$-element set $S \subseteq \Z/p$ with
\[
  \tau(S, \Z/p) = \frac{p}{k}\,\kappa(S, \Z/p) \geq (1-\delta)\frac{p}{k}\ln k .
\]
By \Cref{prop:heis-dict} the left-hand side is $\diamp[H_p](\C[\Z/p], U_S)$, while the trivial bound is $\lceil p/k \rceil \leq (p/k)(1 + k/p)$, and dividing the two estimates gives
\[
  \rr\bigl(\C[\Z/p], U_S\bigr)
  \geq \frac{(1-\delta) \ln k}{1 + k/p}
  \geq (1-\delta)\bigl(1 - o(1)\bigr) \ln p,
\]
since $\ln k = \ln p - \ln \ln p + o(1)$ and $k/p \to 0$.
The right-hand side is at least $(1-\epsilon)\ln p$ for all large $p$.
For the consequence, $\dim \C[\Z/p] = p$, so $\cc(\C[\Z/p]) \geq (1-\epsilon)\ln p$ eventually for every $\epsilon > 0$, while \Cref{prop:cxupper} gives $\cc(\C[\Z/p]) \leq 2 + \ln p$, as desired.
\end{proof}

\subsection{Two-transitive groups of small order}\label{subsec:fin-2trans}

In the previous subsection the action of the Heisenberg group on the subspaces $U_S$ was essentially that of $\Z/p$ translating its subsets.
Such an action is very rigid, so it is not surprising that the covering number can be large.
There are permutation groups that offer far more translates of a given set, and $2$-transitive groups are the extreme case.
We show that, perhaps surprisingly, these groups typically give rise to large complexity as well.

Let $G$ act $2$-transitively on a finite set $\Omega$ with $|\Omega| = n$, and let $V \leq \C[\Omega]$ be the space of functions summing to zero, the \emph{deleted permutation module}, of dimension $n-1$.
It is irreducible, and this is exactly $2$-transitivity.
For $S \subseteq \Omega$ with $|S| \geq 2$ put
\[
  U_S = \left\{ f \in V \mid f = 0 \text{ outside } S \right\},
\]
of dimension $|S| - 1$ by the single relation $\sum_{x \in S}f(x) = 0$, and note that $g U_S = U_{gS}$ for every $g \in G$.
Every element of $\sum_{i \leq m} U_{g_i S}$ vanishes outside $\bigcup_{i \leq m} g_i S$.
If some point $x$ is missing from that union, then $V$ still contains a function nonzero at $x$, for instance $\mathbf{1}_{\{x\}} - \mathbf{1}_{\{y\}}$ for any $y \neq x$, so the sum is not all of $V$.
Covering $\Omega$ is therefore necessary for spanning $V$, and
\begin{equation}\label{eq:cover-lower}
  \diamp[G](V, U_S) \geq \min\left\{ m \Bigm| \bigcup_{i \leq m} g_i S = \Omega \ \text{for some } g_1, \ldots, g_m \in G \right\}.
\end{equation}
Unlike in the Heisenberg example, covering is not also sufficient for spanning, but the one-sided estimate is all we need.

\begin{theorem}\label{thm:2trans}
Let $(G_i, \Omega_i)$ be a sequence of finite $2$-transitive permutation groups of degrees $n_i \to \infty$ with $\ln |G_i| = n_i^{o(1)}$, and let $V_i$ be the deleted permutation modules.
Then
\[
  \cc(V_i) = (1 + o(1)) \ln \dim V_i.
\]
\end{theorem}

\begin{proof}
Fix a constant $0 < c < 1$, put $\alpha = 1/\ln n$ and $m = \lfloor c \ln^2 n \rfloor$, and take $n$ large enough that $0 < \alpha \leq 1/2$ and $m \geq 1$.
Let $S \subseteq \Omega$ be a random subset in which each point is included independently with probability $\alpha$.

We first bound the probability that some $m$ translates of $S$ cover $\Omega$.
Fix $g_1, \ldots, g_m \in G$ and choose points $x_1, x_2, \ldots \in \Omega$ so that the preimage sets $P_j = \{g_i^{-1} x_j \mid i \leq m\}$ are pairwise disjoint.
Once $x_1, \ldots, x_{j-1}$ are chosen, the set $P_j$ meets some earlier $P_{j'}$ exactly when $x_j = g_ag_b^{-1}x_{j'}$ for some $a, b \leq m$ and $j' < j$, which forbids at most $(j-1)m^2$ points.
The process therefore yields
\[
  r = \Bigl\lceil \frac{n}{m^2} \Bigr\rceil
\]
points, because $(r-1)m^2 < n$.
Each $P_j$ has at most $m$ elements, and $S$ misses $P_j$ with probability
\[
  (1-\alpha)^{|P_j|} \geq (1-\alpha)^m \geq \exp\bigl(-m(\alpha + \alpha^2)\bigr) \geq e^{-1}n^{-c},
\]
using $-\ln(1-\alpha) \leq \alpha + \alpha^2$ for $0 \leq \alpha \leq 1/2$ together with
\[
  m(\alpha + \alpha^2) \leq c\ln^2 n\Bigl(\frac{1}{\ln n} + \frac{1}{\ln^2 n}\Bigr) = c \ln n + c \leq c\ln n + 1 .
\]
These events are independent because the sets $P_j$ are pairwise disjoint, and $x_j$ is covered by $\bigcup_{i \leq m} g_i S$ exactly when $P_j$ meets $S$.
Hence
\[
  \Prob\Bigl[\bigcup_{i \leq m} g_i S = \Omega\Bigr]
  \leq \bigl(1 - e^{-1} n^{-c}\bigr)^r
  \leq \exp\Bigl(-\frac{n^{1-c}}{3 m^2}\Bigr),
\]
using $r \geq n/m^2$ and $e^{-1} > 1/3$ in the last step, and a union bound over the at most $|G|^m$ tuples bounds the probability that some $m$ translates cover by $\exp(m \ln|G| - n^{1-c}/(3m^2))$.
Since $m \leq \ln^2 n$, the exponent tends to $-\infty$ whenever $\ln^6 n \cdot \ln|G| = o(n^{1-c})$, which the hypothesis $\ln|G_i| = n_i^{o(1)}$ grants for every fixed $c < 1$.

Meanwhile the Chernoff bound gives $|S| \geq (1 - o(1)) \alpha n$ with probability tending to one.
The two high-probability events occur together, so for all large $n$ some $S$ has $\dim U_S = |S| - 1 \geq (1 - o(1)) n/\ln n$ and admits no covering by $m$ translates.
Hence $\diamp[G](V, U_S) > m$.
The trivial bound is at most $(\dim V)/(\dim U_S) + 1 \leq (1 + o(1)) \ln n$, and therefore
\begin{equation}
\label{eq:2trans-lower}
  \cc(V) \geq \rr(V, U_S) \geq \frac{m}{(1 + o(1)) \ln n} \geq c \bigl(1 - o(1)\bigr) \ln n .
\end{equation}
The constant $c$ has been fixed throughout this argument, and so has the family $(G_i, \Omega_i)$, so \eqref{eq:2trans-lower} says exactly that
\[
  \liminf_{i \to \infty} \frac{\cc(V_i)}{\ln n_i} \geq c .
\]
The left-hand side above does not involve $c$, and the estimate holds for every $c < 1$, so taking the supremum over such $c$ gives $\liminf_{i \to \infty} \cc(V_i)/\ln n_i \geq 1$.
Finally, since $\dim V_i = n_i - 1$ and \Cref{prop:cxupper} bounds $\cc(V_i)$ by $2 + \ln \dim V_i$, we conclude $\cc(V_i) = (1 + o(1)) \ln \dim V_i$, as desired.
\end{proof}

We record two concrete families to which the theorem applies.
The first is affine: a $2$-transitive group of affine transformations of a vector space of order $n = \ell^e$ is a subgroup of $\AGL_e(\F_\ell)$, so
\[
  |G| \leq n \, |\GL_e(\F_\ell)| < n \ell^{e^2} = n^{e+1},
\]
and $e \leq \log_2 n$ gives $\ln |G| = O(\ln^2 n)$.
The smallest instance is $\AGL_1(\F_\ell)$, sharply $2$-transitive on $n = \ell$ points with $|G| = \ell(\ell - 1) < n^2$.
The second is projective: the groups $\PSL_d(\F_q)$ act $2$-transitively on the $n = (q^d - 1)/(q - 1)$ points of $\mathbf{P}^{d-1}(\F_q)$, and here $|G| < q^{d^2}$ and $n \geq q^{d-1}$, so that
\[
  \ln|G| \leq d^2 \ln q \leq \frac{d^2}{d-1}\ln n \leq 2d \ln n = O(\ln^2 n).
\]
Along every such family $\cc(V_i) = (1 + o(1)) \ln \dim V_i$.

By the classification of the finite $2$-transitive groups \cite[Section 7.7]{DM96}, a consequence of the classification of finite simple groups, every such group is either affine or almost simple, with socle an alternating group in its natural action, a simple group of Lie type acting on a projective space or a related geometry, or one of finitely many sporadic examples.
For $\PSL_2(\F_p)$ the module in question is the Steinberg representation $W_p$, of dimension $p$.
Every irreducible representation $V$ of the complex group has $\cc_{\SL_2(\C)}(V) = 1$, while its finite version satisfies
\[
  \cc_{\PSL_2(\F_p)}(W_p) = (1 + o(1))\ln p,
\]
which is as complex as \Cref{prop:cxupper} permits.
The natural actions of the symmetric and alternating groups, of order $n^{n(1+o(1))}$, escape the hypothesis, and there the conclusion genuinely fails, since \Cref{cor:sn-std} gives $\cc_{S_n}(\C^{n-1}) = 1$.

\section{Bounded complexity}\label{sec:bounded}

In this section we give examples of connected groups where the covering complexity remains bounded along a growing family of representations.
The method follows the two-stage scheme of \Cref{subsec:expansion}.
The growth process of \Cref{lem:growth} produces a subspace of large dimension, and a structural property of the specific representation then covers the whole space with a bounded number of further translates.

\subsection{Reduction to large subspaces}\label{subsec:reduction}

\begin{proposition}\label{prop:reduction}
Let $H$ act irreducibly on $V$ as in \Cref{def:complexity}, either as a compact group acting unitarily or as a complex reductive group acting rationally, and suppose there are constants $0 \leq \theta < 1$ and $C_0 \geq 1$ such that
\[
  \dim W \geq \theta \dim V \quad \Longrightarrow \quad \diamp[H](V, W) \leq C_0
\]
for every nonzero subspace $W \leq V$. Then
\[
  \cc(V) \leq C_0 \max\Bigl\{1, \ \ln \frac{1}{1-\theta} + \theta\Bigr\}.
\]
\end{proposition}

\begin{proof}
Suppose first that $H$ is a complex reductive group and let $K \leq H$ be a maximal compact subgroup.
Then $V$ is irreducible as a representation of $K$, as in the proof of \Cref{cor:reductive}, and $\diamp[H](V, W) = \diamp[K](V, W)$ for every $W \leq V$ by \Cref{lem:compact-algebraic}.
Both the hypothesis and the conclusion are therefore unchanged when $H$ is replaced by $K$, and we may assume from now on that $H$ is compact and acts unitarily.

Let $U \leq V$ be nonzero and put $\alpha = \dim U/\dim V$.
If $\alpha \geq \theta$ then the hypothesis applies to $U$ itself, and
\[
  \rr(V, U) = \frac{\diamp[H](V, U)}{\lceil 1/\alpha\rceil} \leq \alpha \diamp[H](V, U) \leq \alpha C_0 \leq C_0 .
\]
So assume $\alpha < \theta$ and run the growth process of \Cref{lem:growth} for $j = \lceil \alpha^{-1}\ln (1-\theta)^{-1}\rceil$ steps, which produces a sum $W_j$ of $j$ translates of $U$.
By the estimate \eqref{eq:decay},
\[
  \dim W_j^{\perp} < \dim V \cdot e^{-\alpha j} \leq (1-\theta)\dim V,
\]
so $\dim W_j > \theta \dim V$ and the hypothesis applies to $W_j$.
Write $m = \diamp[H](V, W_j) \leq C_0$.
Each of the $m$ translates of $W_j$ that span $V$ is itself a sum of $j$ translates of $U$, so $mj \leq C_0 j$ translates of $U$ suffice and
\[
  \rr(V, U) \leq \alpha \diamp[H](V, U) \leq \alpha C_0 j
  \leq C_0 \Bigl(\ln \frac{1}{1-\theta} + \alpha\Bigr)
  < C_0 \Bigl(\ln \frac{1}{1-\theta} + \theta\Bigr).
\]
Taking the larger of the two cases gives the stated maximum, as desired.
\end{proof}

The hypothesis of the proposition can be tested on Borel-stable subspaces alone.
The proof of \cite[Proposition 2.3]{JS25} gives the following slightly more explicit formulation.

\begin{lemma}\label{lem:borel}
Let $G$ be a connected linear algebraic group with a rational representation on $V$ and let $W \leq V$.
Then there is a subspace $W_0 \leq V$ stable under a Borel subgroup of $G$, with $\dim W_0 = \dim W$ and $\diamp[G](V, W) \leq \diamp[G](V, W_0)$.
\end{lemma}

\subsection{Conjugation of \texorpdfstring{$\SL_n(\C)$}{SL(n,C)} on \texorpdfstring{$\slalg_n(\C)$}{sl(n,C)}}

We first apply the reduction to the conjugation representation of $\SL_n(\C)$ on its Lie algebra $\slalg_n(\C)$.
Here we rely on \cite[Proposition 5.2]{JS25}, which bounds the diameter of upper right block closed subspaces of $\slalg_n(\C)$ of dimension greater than $3 n^2/4 + n/2$ by $8$.
By \cite[Section 2.2.2]{JS25}, every Borel-stable subspace of $\slalg_n(\C)$ is upper right block closed, so the proposition applies to them.
The bound there is stated for conjugation by $\GL_n(\C)$, but scalars act trivially, so it is the same as the $\SL_n(\C)$ diameter.

\begin{proposition}\label{prop:cxadjoint}
For every $n \geq 2$ we have $\cc(\slalg_n(\C)) < 18$.
\end{proposition}
\begin{proof}
Take $\theta = 3/4 + 2.7 \cdot 10^{-4}$.
Then
\[
\theta (n^2 - 1) > 3 n^2 / 4 + n/2
\quad
\text{ for all } n > 1902,
\]
so we can take $C_0 = 8$ in \Cref{prop:reduction} for all $n > 1902$, and the complexity is bounded by $C_0 \max\{1, \ln(1/(1-\theta)) + \theta\} < 18$.
For $2 \leq n \leq 1902$ we can use the trivial bound $\cc(\slalg_n(\C)) \leq 2 + \ln(n^2 - 1) \leq 2 + \ln(1902^2 - 1) < 18$, and the claim follows.
\end{proof}

\subsection{Plane curves and the representations \texorpdfstring{$\Sym^k \C^3$}{Sym\^{}k C\^{}3}}
We now carry out the criterion for $V = \Sym^k \C^3$.
Two inputs are needed.
The first is combinatorial and says that a large Borel-stable subspace contains a big block of monomials.
The second is a statement about plane curves with imposed multiplicities at given points.

Throughout, $E = \C^3$ with basis $e_1, e_2, e_3$, we take $B$ to be the upper triangular Borel subgroup of $\SL_3(\C)$, and we write $p = [e_1] \in \mathbf{P}(E)$ and $N = \dim \Sym^k E = \binom{k+2}{2}$.

\begin{lemma}\label{lem:block}
Let $W_0 \leq \Sym^k E$ be a nonzero $B$-stable subspace.
Then there is an integer $\kappa$ with $0 \leq \kappa \leq k$ and
\[
  W_0 \supseteq e_1^{\,k-\kappa}\Sym^{\kappa}E
  \qquad \text{and} \qquad
  \dim W_0 \leq (\kappa+1)(k+1).
\]
\end{lemma}

\begin{proof}
The diagonal torus $T \leq B$ acts semisimply, and its weight spaces in $\Sym^k E$ are the lines spanned by the monomials $e_1^{\,k-b-c}e_2^{\,b}e_3^{\,c}$.
These weights are pairwise distinct, so a $T$-stable subspace is the sum of some of these lines, and in particular
\[
  W_0 = \bigl\langle e_1^{\,k-b-c}e_2^{\,b}e_3^{\,c} \bigm| (b,c) \in \mathcal{A} \bigr\rangle
\]
for a set $\mathcal{A}$ of pairs.
Being $B$-stable, $W_0$ is also stable under the positive root vectors $e_1\partial/\partial e_2$, $e_1\partial/\partial e_3$ and $e_2\partial/\partial e_3$, which on the monomial indexed by $(b,c)$ make the moves
\[
  (b,c) \mapsto (b-1,c), \qquad
  (b,c) \mapsto (b,c-1), \qquad
  (b,c) \mapsto (b+1,c-1)
\]
whenever the relevant exponent is positive.
The first two moves decrease $b+c$ by one and the third decreases $c$ by one while fixing $b+c$, so composing them shows that $(b,c) \in \mathcal{A}$ implies $(b',c') \in \mathcal{A}$ for every $b', c' \geq 0$ with $c' \leq c$ and $b'+c' \leq b+c$.

Fix $c$ and consider the slice $\mathcal{A}_c = \{b \mid (b,c) \in \mathcal{A}\}$.
The first move shows that it is either empty or an initial segment $\{0, \ldots, m_c\}$.
If it is nonempty and $c > 0$, the third move applied to $(m_c, c)$ gives $(m_c + 1, c-1) \in \mathcal{A}$, so $\mathcal{A}_{c-1}$ is nonempty and $m_{c-1} \geq m_c + 1$.
Let $\kappa$ be the largest $c$ with $\mathcal{A}_c \neq \emptyset$, which exists because $W_0$ is nonzero, and which satisfies $\kappa \leq k$ because $c$ is an exponent in a monomial of degree $k$.
Only the slices $\mathcal{A}_0, \ldots, \mathcal{A}_\kappa$ can then be nonempty, each has at most $k+1$ elements, and the dimension bound follows.

It remains to see that the slices are long enough to contain a full triangle, and the point is that one quantity is monotone.
Put $s_c = m_c + c$.
The inequality $m_{c-1} \geq m_c + 1$ says exactly that $s_c$ is non-increasing in $c$, so $s_c \geq s_\kappa = m_\kappa + \kappa \geq \kappa$ for every $0 \leq c \leq \kappa$, and hence $m_c \geq \kappa - c$.
Therefore $\mathcal{A}$ contains every $(b,c)$ with $b + c \leq \kappa$, and the corresponding monomials are
\[
  e_1^{\,k-b-c}e_2^{\,b}e_3^{\,c}
  = e_1^{\,k-\kappa}\bigl(e_1^{\,\kappa-b-c}e_2^{\,b}e_3^{\,c}\bigr),
\]
which span exactly $e_1^{\,k-\kappa}\Sym^\kappa E$.
\end{proof}

Under the identification of $(\Sym^k E)^*$ with $\C[x_1,x_2,x_3]_k$ used in \Cref{prop:sym}, the pairing of $x^\beta$ with $e^\alpha$ vanishes unless $\alpha = \beta$.
The block consists of the monomials whose $e_1$-exponent is at least $k - \kappa$, so its annihilator is spanned by the dual monomials with $x_1$-exponent at most $k-\kappa-1$, equivalently with total $(x_2,x_3)$-degree at least $\kappa+1$.
In the affine chart $x_1 \neq 0$ around $p = [1:0:0]$ the multiplicity of a homogeneous form $F$ at $p$ is the least total $(x_2,x_3)$-degree occurring in $F(1,x_2,x_3)$, and therefore
\[
  \bigl(e_1^{\,k-\kappa}\Sym^\kappa E\bigr)^{\perp}
  = \langle x^\beta \mid \beta_1 \leq k-\kappa-1 \rangle
  = \{F \in \C[x_1,x_2,x_3]_k \mid \operatorname{mult}_p F \geq \kappa+1\}.
\]
The description is equivariant, in that the annihilator of a translate of the block is the space of degree $k$ forms vanishing to the same order at the corresponding translate of $p$.
So a translate $gW_0$ fails to span only through forms of multiplicity at least $\kappa+1$ at the point $gp$, and since $\SL_3(\C)$ is transitive on $\mathbf{P}^2$ we may put those points wherever we like, in particular in a grid.

\begin{lemma}\label{lem:grid}
Let $s \geq 2$ and $\rho \geq 1$ with $s\rho > k$, let $a_1, \ldots, a_s$ and $b_1, \ldots, b_s$ be distinct scalars, and let $\Gamma = \{(a_i, b_j)\} \subset \mathbf{A}^2 \subset \mathbf{P}^2$ be the resulting grid of $s^2$ points.
Then no nonzero $F \in \C[x_1,x_2,x_3]_k$ satisfies $\operatorname{mult}_q F \geq \rho$ for all $q \in \Gamma$.
\end{lemma}

\begin{proof}
Write $F = x_1^tG$ with $x_1 \nmid G$ and put $f(x_2,x_3) = G(1,x_2,x_3)$, a nonzero polynomial of degree $d \leq k$.
Multiplication by a power of $x_1$, which is a unit on the chart $x_1 \neq 0$, changes no local multiplicity there, so $\operatorname{mult}_q f = \operatorname{mult}_q F \geq \rho$ at every $q \in \Gamma$, and $s\rho > k \geq d$.
It therefore suffices to show that no nonzero polynomial $f$ of degree $d$ has multiplicity at least $\rho$ at every point of $\Gamma$ when $s\rho > d$.
Suppose that such an $f$ exists.

Restrict $f$ to the vertical line $x_2 = a_i$.
Restricting a function lying in the $\rho$-th power of the maximal ideal at a point gives either zero or a function vanishing there to order at least $\rho$, so unless the restriction is zero it is a univariate polynomial of degree at most $d$ with $s$ roots each of multiplicity at least $\rho$.
That would give it at least $s\rho > d$ roots with multiplicity, so the restriction vanishes identically.
Hence $x_2 - a_i$ divides $f$ for every $i$, and the same argument on horizontal lines shows that $x_3 - b_j$ divides $f$ for every $j$.
As these $2s$ linear forms are pairwise coprime,
\[
  f = \prod_{i \leq s}(x_2-a_i) \prod_{j \leq s}(x_3-b_j) \cdot f_1
\]
with $f_1 \neq 0$ and $\deg f_1 \leq d - 2s$.
At a grid point exactly one vertical and one horizontal factor vanishes, each to order one, so removing their product lowers the multiplicity by exactly two and $\operatorname{mult}_q f_1 \geq \rho - 2$ for all $q \in \Gamma$.
Both sides of $s\rho > d$ drop by $2s$, so the strict inequality $s(\rho-2) > d - 2s$ persists.

The step may be repeated as long as the required residual multiplicity stays positive, and we finish by parity.
If $\rho = 2q$ is even, then after $q$ iterations we hold a nonzero polynomial of degree at most $d - 2sq < 0$, which is absurd.
If $\rho = 2q+1$ is odd, then after $q$ iterations we hold a nonzero $f_q$ of degree at most $d - 2sq < s$ vanishing at every grid point.
Restricting it to each vertical line gives $s$ distinct roots in the $x_3$-variable, so every such restriction vanishes and $\prod_i(x_2-a_i)$ divides $f_q$, impossible because that product has degree $s > \deg f_q$.
\end{proof}

We now have enough to run the criterion.

\begin{theorem}\label{thm:cxsym}
For every $k \geq 1$ and every nonzero subspace $U \leq \Sym^k \C^3$,
\[
  \diamp[\SL_3(\C)]\bigl(\Sym^k \C^3, U\bigr)
  \leq 9\ln 3 \cdot \frac{\dim \Sym^k \C^3}{\dim U} + 9,
\]
and hence $\cc(\Sym^k \C^3) < 16$.
\end{theorem}

\begin{proof}
We verify the hypothesis of \Cref{prop:reduction} with $\theta = 2/3$ and $C_0 = 9$, so let $W \leq \Sym^k E$ have $\dim W \geq 2N/3$.
\Cref{lem:borel} produces a subspace $W_0$ of the same dimension, stable under some Borel subgroup $B'$, with $\diamp[\SL_3](\Sym^k E, W) \leq \diamp[\SL_3](\Sym^k E, W_0)$.
All Borel subgroups of $\SL_3$ are conjugate, and replacing $W_0$ by one overall translate changes neither its dimension nor its diameter, so we may assume that $W_0$ is stable under the fixed upper triangular $B$ used above.

Let $\kappa$ be as in \Cref{lem:block} and put $\rho = \kappa + 1$.
From $(\kappa+1)(k+1) \geq \dim W_0 \geq 2N/3 = (k+1)(k+2)/3$ we get
\[
  3\rho \geq k+2 > k.
\]
Choose a $3 \times 3$ grid $\Gamma \subset \mathbf{A}^2 \subset \mathbf{P}^2$ and, using transitivity of $\SL_3(\C)$ on $\mathbf{P}^2$, elements $g_1, \ldots, g_9$ with $g_ip$ running once over $\Gamma$.
Write $L_\kappa = e_1^{\,k-\kappa}\Sym^\kappa E$, so that $L_\kappa \subseteq W_0$ by \Cref{lem:block} and therefore
\begin{align*}
  \Bigl(\sum_{i \leq 9}g_iW_0\Bigr)^{\perp}
  &\subseteq \Bigl(\sum_{i \leq 9}g_iL_\kappa\Bigr)^{\perp}
  = \bigcap_{i \leq 9}(g_iL_\kappa)^{\perp} \\
  &= \bigl\{F \in \C[x_1,x_2,x_3]_k \bigm| \operatorname{mult}_{g_ip}F \geq \rho \text{ for all } i\bigr\}.
\end{align*}
This vanishes by \Cref{lem:grid} applied with $s = 3$, so the nine displayed translates of $W_0$ span, and \Cref{lem:borel} carries the bound back to $W$.
Every subspace of dimension at least $2N/3$ therefore has diameter at most $9$.

\Cref{prop:reduction} now gives
\[
  \cc(V) \leq 9\Bigl(\ln 3 + \frac{2}{3}\Bigr) = 15.887\ldots < 16 .
\]
For the diameter itself, put $\alpha = \dim U/N$.
If $\alpha < 2/3$, the growth process of \Cref{lem:growth} run for $j = \lceil \alpha^{-1}\ln 3\rceil$ steps produces a subspace of dimension greater than $2N/3$, nine translates of which span $V$, at a total cost of at most
\[
  9j \leq \frac{9\ln 3}{\alpha} + 9 .
\]
If $\alpha \geq 2/3$ then nine translates of $U$ span outright and the same bound holds trivially.
Substituting $\alpha = \dim U/N$ gives the stated estimate, as desired.
\end{proof}

By \Cref{prop:sym} with $n = 3$ the complexity is at least $2(k+1)/(k+3)$, so along this family it is asymptotically at least $2$ and always less than $16$.
Applying the outer automorphism of $\SL_3(\C)$ gives the same bounds for the representations $V(0,k)$.

\begin{remark}
The choice $\theta = 2/3$ is not arbitrary.
A grid of $s^2$ points costs $C_0 = s^2$ translates, and \Cref{lem:grid} needs $s\rho > k$ while \Cref{lem:block} only guarantees $\rho \geq \theta(k+2)/2$, so the two are compatible for large $k$ precisely when $s \geq 2/\theta$.
\Cref{prop:reduction} then gives the bound $s^2\bigl(\ln(1-\theta)^{-1} + \theta\bigr)$, which increases with $\theta$, so for each $s$ the best admissible choice is $\theta = 2/s$.
Now $s = 2$ would force $\theta \geq 1$ and is excluded, $s = 3$ with $\theta = 2/3$ gives $9(\ln 3 + 2/3) = 15.887\ldots$, and $s = 4$ with $\theta = 1/2$ already gives $16(\ln 2 + 1/2) = 19.09\ldots$, with larger grids worse still.
\end{remark}

\subsection{Questions}

We have bounded the complexity of the representations $\Sym^k \C^3$ of $\SL_3(\C)$.
We have not been able to extend the argument to the other irreducible representations $V(a,b)$ of $\SL_3(\C)$.
The same method would work provided one could show that large Borel-stable subspaces of $V(a,b)$ have bounded diameter.

\begin{question}\label{q:missing}
Are there $\theta < 1$ and $C_0$ such that for every pair $(a,b)$ and every Borel-stable subspace $W_0 \leq V(a,b)$ with $\dim W_0 \geq \theta \dim V(a,b)$ one has $\diamp[\SL_3(\C)](V(a,b), W_0) \leq C_0$?
\end{question}

Beyond $\SL_3(\C)$, \Cref{cor:flagbound} relates the complexity of representations of a reductive group to the dimension of its flag variety.
All examples we have explored are consistent with a positive answer to the following question.

\begin{question}\label{q:fixedgroup}
Let $G$ be a connected reductive group over $\C$.
Is $\cc(V)$ bounded over all irreducible rational representations $V$ of $G$?
If so, can one bound it by $\max ( 1, \dim G/B )$?
\end{question}

\section{Monomial diameters in the Lie algebra}\label{sec:lie}

In this section we consider the same problem as before, but with a Lie algebra $\lalg$ acting on $V$ instead of a group.
We adapt the definitions and questions from \cite[Section 6]{JS25}.
Let $\lalg$ be a Lie algebra with a representation $\rho$ on a finite-dimensional complex vector space $V$, and let $\mon(\rho(\lalg))$ be the set of \emph{monomials}
\[
  \rho(x_1)\rho(x_2)\cdots\rho(x_r),
  \qquad x_1, \ldots, x_r \in \lalg, \quad r \geq 0,
\]
the identity operator being the empty product.
The \emph{monomial Lie-additive diameter} of $V$ with respect to a nonzero $U \leq V$ is
\[
  \diamp[\lalg,\mon](V, U)
  = \min\bigl\{d \mid m_1 \cdot U + \cdots + m_d \cdot U = V
  \text{ for some } m_i \in \mon(\rho(\lalg))\bigr\},
\]
the least number of monomial translates of $U$ whose sum is all of $V$.
Since $\dim m_i \cdot U \leq \dim U$, the trivial lower bound of \Cref{sec:intro} applies here as well,
\[
  \diamp[\lalg,\mon](V, U) \geq \Bigl\lceil \frac{\dim V}{\dim U} \Bigr\rceil.
\]
In \cite[Question 6.6]{JS25} we asked whether the monomial Lie-additive diameter is always equal to the trivial lower bound.
Here we offer a new family of examples where the answer is positive.
These examples are also in line with \cite[Question 6.7]{JS25}, which asks whether the monomial Lie-additive diameter is always at most the group-additive diameter.

\subsection{Symmetric powers of a plane in three variables}

The group-additive diameter of the pair $(\Sym^k \C^3, \Sym^k X)$ with $X$ a plane in $\C^3$ is $k+1$ by \Cref{prop:sym}, while the trivial bound is $\lceil (k+2)/2\rceil$. What makes the situation interesting is that the trivial bound is attained in the Lie algebra.

\begin{theorem}\label{thm:liesym}
Let $\slalg_3(\C)$ act on $V = \Sym^k \C^3$ and let $U = \Sym^k X$ for a plane $X \leq \C^3$. Then
\[
  \diamp[\slalg_3(\C),\mon](V, U) = \Bigl\lceil \frac{k+2}{2} \Bigr\rceil
\]
for every $k \geq 1$, and in a basis with $X = \langle e_1, e_2\rangle$ the bound is attained by monomials whose factors are all of the form $\alpha E_{21} + \beta E_{32}$.
\end{theorem}

Put $d = \lceil (k+2)/2\rceil$ for the rest of the subsection.
Since $\dim V = \binom{k+2}{2}$ and $\dim U = k+1$, the trivial bound is $\lceil (k+2)/2 \rceil = d$, and this is the lower bound of the theorem.
The upper bound occupies the rest of the subsection, and the plan is the following.
A grading splits the problem into $2k+1$ separate spanning conditions, one for each graded piece.
We count the monomials available in each piece (\Cref{lem:liecount}), we compute exactly which vectors a single monomial can reach there (\Cref{lem:liereach}), and we then fill each piece greedily along a flag (\Cref{lem:flag}) before making the finitely many choices compatible.

\emph{The setup.}
It is enough to treat the coordinate plane.
Indeed, $\SL_3(\C)$ is transitive on the planes in $\C^3$, so some $g \in \SL_3(\C)$ carries $\langle e_1, e_2\rangle$ to $X$, and conjugation by $\rho(g)$ sends each factor $\rho(x)$ to $\rho(\Ad(g)x)$ and therefore sends monomials to monomials of the same length.
The monomial diameter is thus unchanged when $U$ is replaced by $\rho(g)U$.
So identify $V$ with $\C[e_1,e_2,e_3]_k$ in such a way that $\rho(E_{ij}) = e_i \partial/\partial e_j$, and take $X = \langle e_1, e_2 \rangle$, so that $U = \C[e_1,e_2]_k$.
Let
\[
  \aalg = \langle E_{21}, E_{32}\rangle \leq \slalg_3(\C),
\]
a two-dimensional vector subspace, not a Lie subalgebra, whose role is that it is the degree $-1$ piece for the grading introduced next.
Let $m_0$ be the identity and, for $1 \leq i \leq d-1$, let
\[
  m_i = \rho(y_{i1})\rho(y_{i2})\cdots\rho(y_{i,2i}),
  \qquad y_{il} \in \aalg .
\]
What has to be shown is that the parameters $y_{il}$ can be chosen so that $m_0 \cdot U + \cdots + m_{d-1} \cdot U = V$.

Let the one-parameter subgroup $t \mapsto \diag(t, 1, t^{-1})$ act on $V$ and write $V = \bigoplus_{\theta = -k}^{k}V^{(\theta)}$, so that $V^{(\theta)}$ is spanned by the monomials $e_1^ae_2^be_3^c$ with $a - c = \theta$.
Putting
\[
  v_{\theta,c} = e_1^{\,\theta+c}e_2^{\,k-\theta-2c}e_3^{\,c},
  \qquad
  \mu = \max(0,-\theta),
  \qquad
  \nu = \Bigl\lfloor \frac{k-\theta}{2} \Bigr\rfloor,
\]
the monomials of $V^{(\theta)}$ are exactly the $v_{\theta,c}$ with $\mu \leq c \leq \nu$, and hence $\dim V^{(\theta)} = \nu - \mu + 1$.
Both $E_{21}$ and $E_{32}$ have degree $-1$ for this grading, since $E_{21}$ replaces one copy of $e_1$ by $e_2$ while $E_{32}$ replaces one copy of $e_2$ by $e_3$, and either operation lowers $a - c$ by one.
Each $m_i$ is therefore homogeneous of degree $-2i$ and carries $V^{(\theta')}$ into $V^{(\theta'-2i)}$.
The subspace $U$ is spanned by the monomials $\xi_a = e_1^ae_2^{\,k-a}$ with $0 \leq a \leq k$, and $\xi_a \in V^{(a)}$, so $U$ meets every graded piece in at most one basis vector.
Write $I(\theta)$ for the set of indices $i$ with $0 \leq i \leq d-1$ and $0 \leq \theta+2i \leq k$, which are those contributing to the grade $\theta$, and for $i \in I(\theta)$ put
\begin{equation}\label{eq:reachable}
  S_i(\theta) = \bigl\langle v_{\theta,c} \bigm| \mu \leq c \leq \min(2i, \nu)\bigr\rangle .
\end{equation}
Everything in sight is graded, and therefore
\begin{equation}\label{eq:gradedsplit}
  \Bigl(\sum_{i < d} m_i \cdot U\Bigr) \cap V^{(\theta)}
  = \bigl\langle m_i \xi_{\theta+2i} \bigm| i \in I(\theta) \bigr\rangle .
\end{equation}
So spanning $V$ amounts to $2k+1$ separate spanning conditions, one for each $V^{(\theta)}$.
The conditions are not independent, since they involve the same parameters, and the genericity argument at the end is what makes them hold at once.

The first input is a count of how many monomials reach a given grade.

\begin{lemma}\label{lem:liecount}
For every $-k \leq \theta \leq k$ we have $I(\theta) = \{\lceil \mu/2\rceil, \ldots, \min(\nu, d-1)\}$ and
\[
  |I(\theta)| \geq \dim V^{(\theta)} .
\]
\end{lemma}

\begin{proof}
The description of $I(\theta)$ is what $0 \leq \theta + 2i \leq k$ and $0 \leq i \leq d-1$ say.
Put $i_{\min} = \lceil \mu/2\rceil$ and $i_{\max} = \min(\nu, d-1)$, so that the claim reads $i_{\max} - i_{\min} + 1 \geq \nu - \mu + 1$.
Since $\mu - \lceil \mu/2\rceil = \lfloor \mu/2 \rfloor$, this is the same as $\min(\nu, d-1) \geq \nu - \lfloor \mu/2 \rfloor$.
If $\nu \leq d-1$ it is clear.
Otherwise $\nu > d - 1 = \lceil k/2\rceil$, which forces $\theta < 0$ and hence $\mu = -\theta$, and the parity inequality
\[
  \nu = \Bigl\lfloor \frac{k+\mu}{2} \Bigr\rfloor
  \leq \Bigl\lceil \frac{k}{2} \Bigr\rceil + \Bigl\lfloor \frac{\mu}{2} \Bigr\rfloor
\]
gives the claim.
\end{proof}

The second input identifies exactly what one monomial can reach as its parameters vary.
The subspace \eqref{eq:reachable} grows with $i$, so a longer monomial reaches further into the graded piece.

\begin{lemma}\label{lem:liereach}
Fix $-k \leq \theta \leq k$ and $i \in I(\theta)$.
As the parameters $y_{i1}, \ldots, y_{i,2i}$ range over $\aalg$, the vectors $m_i\xi_{\theta+2i}$ span $S_i(\theta)$.
\end{lemma}

\begin{proof}
Write $A = \rho(E_{21})$ and $B = \rho(E_{32})$ and expand each factor as $\rho(y_{il}) = \alpha_{il}A + \beta_{il}B$.
Multiplying out, $m_i\xi_{\theta+2i}$ is a linear combination of the vectors $W\xi_{\theta+2i}$, one for each word $W$ of length $2i$ in the letters $A$ and $B$, the coefficient of $W$ being the product of one $\alpha_{il}$ or $\beta_{il}$ per letter.
Distinct words carry distinct monomial functions of the parameters.
Consequently, if a linear functional on $V^{(\theta)}$ vanishes on every achievable vector, then applying it to the expansion gives a polynomial in the parameters that vanishes identically, so every coefficient of that polynomial is zero and the functional kills every $W\xi_{\theta+2i}$.
The achievable vectors therefore span the same subspace as the word vectors themselves, and it is enough to identify the span of the latter.

Now $B$ raises the degree in $e_3$ by one and $A$ does not change it, so a word with $q$ letters $B$ carries $\xi_{\theta+2i}$ into $\C v_{\theta,q}$.
Taking the word $B^qA^{2i-q}$, whose rightmost letters act first,
\[
  B^qA^{2i-q}\xi_a
  = \frac{a!}{(a-2i+q)!} \cdot \frac{(k-a+2i-q)!}{(k-a+2i-2q)!} \cdot v_{\theta,q},
  \qquad a = \theta + 2i,
\]
which is nonzero exactly when $2i - q \leq a$ and $2q \leq k - \theta$, that is when $\mu \leq q \leq \min(2i,\nu)$.
Conversely any nonzero word with $q$ letters $B$ produces a scalar multiple of $v_{\theta,q}$, and nonnegativity of its final exponents forces the same two inequalities.
The word vectors therefore span exactly $S_i(\theta)$.
\end{proof}

The third input is the elementary fact that makes the one vector at a time filling work.
It says that a flag growing at least one dimension per step can be filled by picking one vector at a time, no matter which spanning set each step offers.

\begin{lemma}\label{lem:flag}
Let $W$ be a finite-dimensional vector space and let $S_0 \subseteq S_1 \subseteq \cdots \subseteq S_M$ be subspaces of $W$ with
\[
  \dim S_r \geq \min\bigl(r+1, \dim W\bigr) \quad \text{for } 0 \leq r \leq M,
  \qquad M + 1 \geq \dim W .
\]
For each $r$ let $T_r \subseteq S_r$ be a subset spanning $S_r$.
Then there are $w_0 \in T_0, \ldots, w_M \in T_M$ with $\langle w_0, \ldots, w_M\rangle = W$.
\end{lemma}

\begin{proof}
Choose the $w_r$ in the order $r = 0, 1, \ldots, M$.
Suppose $w_0, \ldots, w_{r-1}$ have been chosen and let $P$ be their span, so that $\dim P \leq r$.
If $P = W$, take $w_r \in T_r$ arbitrarily.
Otherwise $\dim P < \dim W$ and $\dim P \leq r$, so
\[
  \dim P < \min\bigl(r+1, \dim W\bigr) \leq \dim S_r,
\]
and since $T_r$ spans $S_r$ some element of $T_r$ lies outside $P$.
Take that element for $w_r$.
The dimension of the accumulated span therefore grows by one at every step until it reaches $\dim W$, and it does reach it because there are $M+1 \geq \dim W$ steps.
\end{proof}

\begin{proof}[Proof of \Cref{thm:liesym}]
The lower bound was noted above, so only the upper bound remains.

Fix $\theta$ and list $I(\theta) = \{i_0 < i_1 < \cdots < i_M\}$, a run of consecutive integers with $i_0 = \lceil \mu/2\rceil$ by \Cref{lem:liecount}.
The subspaces $S_{i_r}(\theta)$ increase with $r$ and form a flag in $V^{(\theta)}$, and since $2i_0 \geq \mu$,
\[
  \dim S_{i_r}(\theta) = \min(2i_0 + 2r, \nu) - \mu + 1
  \geq \min\bigl(2r + 1, \dim V^{(\theta)}\bigr)
  \geq \min\bigl(r + 1, \dim V^{(\theta)}\bigr),
\]
so the flag grows at least as fast as the steps come.
\Cref{lem:liecount} gives $M + 1 = |I(\theta)| \geq \dim V^{(\theta)}$, and \Cref{lem:liereach} says that the vectors available at step $r$ span $S_{i_r}(\theta)$.
The parameters belonging to different $m_i$ are independent, so \Cref{lem:flag} applies with $W = V^{(\theta)}$ and produces a choice of the $y_{i_r,l}$ for which the vectors $m_{i_r}\xi_{\theta+2i_r}$ span $V^{(\theta)}$.
By \eqref{eq:gradedsplit} this is exactly the spanning condition at the grade $\theta$.

It remains to satisfy all the grades at once.
For fixed $\theta$, the condition just established says that at least one maximal minor of the coordinate matrix of the vectors $m_i\xi_{\theta+2i}$, $i \in I(\theta)$, is nonzero.
The entries are polynomials in the parameters $\alpha_{il}, \beta_{il}$, so the condition cuts out a Zariski open subset $\Omega_\theta$ of the common parameter space.
There are $\sum_{i=1}^{d-1}2i = d(d-1)$ factors in play, each carrying two scalar parameters, so that space is the affine space $\mathbf{A}^{2d(d-1)}$.
The previous paragraph shows that every $\Omega_\theta$ is nonempty, affine space is irreducible, and a finite intersection of nonempty Zariski open subsets of an irreducible variety is nonempty.
Only the $2k+1$ values $-k \leq \theta \leq k$ occur, so $\bigcap_\theta \Omega_\theta$ contains a point.
Any such point gives $\sum_{i<d}m_i \cdot U = V$, and hence $\diamp[\slalg_3(\C),\mon](V,U) \leq d$, and we are done.
\end{proof}

\subsection{Monomial diameters in higher rank}

The first case to inspect beyond \Cref{thm:liesym} is the family from \Cref{prop:sym}, where the group-additive diameter is $k+1$ for every $n$, while the trivial bound can be much smaller.

\begin{question}\label{q:liesln}
Is
\[
  \diamp[\slalg_n(\C),\mon]\bigl(\Sym^k \C^n, \Sym^k \C^{n-1}\bigr)
  = \Bigl\lceil \frac{n+k-1}{n-1} \Bigr\rceil
\]
for all $n \geq 3$ and all $k \geq 1$?
\end{question}

\bibliographystyle{alpha}
\bibliography{biblio}

\end{document}